\documentclass[11pt]{amsart}
\usepackage{egastyle}

\usepackage{amssymb,amsmath,amscd,amsthm,amsfonts,wasysym,mathrsfs,amsxtra,stmaryrd,verbatim,mathtools,soul}
\usepackage[hypertexnames=false]{hyperref} 

\usepackage{graphicx,array,url}
\usepackage[vcentermath]{genyoungtabtikz}

\usepackage{tikz}
\usetikzlibrary{math,matrix,positioning,decorations.markings,arrows,decorations.pathmorphing,
    backgrounds,fit,positioning,shapes.symbols,chains,shadings,fadings,calc}

\usepackage{pgfplots}\pgfplotsset{compat=1.13}
\usetikzlibrary{arrows,positioning,calc,patterns}

\usepackage[section]{placeins}
\usepackage{afterpage}
\input xy
\xyoption{all}

\newcommand{\RR}{\mathbb{R}}
\newcommand{\OO}{\mathcal{O}}

\newcommand{\QQ}{\mathbb{Q}}

\newcommand{\image}{\textnormal{im}\,}

\newcommand{\degree}{\textnormal{deg}\,}

\newcommand{\Hom}{\textnormal{Hom}}
\newcommand{\dimension}{\textnormal{dim}\,}

\newcommand{\rank}{\textnormal{rk}\,}

\newcommand{\Ext}{\textnormal{Ext}}

\newcommand{\Ac}{\mathcal{A}}
\newcommand{\Bc}{\mathcal{B}}
\newcommand{\Cc}{\mathcal{C}}
\newcommand{\Fc}{\mathcal{F}}
\newcommand{\Tc}{\mathcal{T}}

\newcommand{\Coh}{\mathrm{Coh}}

\newcommand{\arinj}{\ar@{^{(}->}}
\newcommand{\arsurj}{\ar@{->>}}
\newcommand{\areq}{\ar@{=}}

\newcommand{\wh}{\widehat}

\newcommand{\Bl}{\mathcal{B}^l}
\newcommand{\ch}{\mathrm{ch}}

\newcommand{\scalea}{\scalebox{0.5}}

\newcommand{\oo}{{\overline{\omega}}}

\newcommand{\whPhi}{{\wh{\Phi}}}

\newcommand{\Stab}{\mathrm{Stab}} 

\begin{document}

\title[Fourier-Mukai transforms and stable  sheaves on elliptic surfaces]{Fourier-Mukai transforms and stable sheaves \\ on Weierstrass elliptic surfaces}

\author[Wanmin Liu]{Wanmin Liu}
\address{Center for Geometry and Physics \\
Institute for Basic Science (IBS)\\
Pohang 37673 \\
Korea}
\address{Current: Department of Mathematics, Uppsala University, Box 480, 751 06 Uppsala, Sweden }
\email{wanminliu@gmail.com}

\author[Jason Lo]{Jason Lo}
\address{Department of Mathematics \\
California State University Northridge\\
18111 Nordhoff Street\\
Northridge CA 91330 \\
USA}
\email{jason.lo@csun.edu}

\author[Cristian Martinez]{Cristian Martinez}
\address{IMECC-UNICAMP\\
Departamento de Matem{\'a}tica\\
Rua S{\'e}rgio Buarque de Holanda, 651\\
Campinas-SP 13083-970, Brazil}
\email{cristian@unicamp.br}

\date{\today}

\begin{abstract}
On a Weierstra{\ss} elliptic surface $X$, we define a `limit' of Bridgeland stability conditions, denoted as $Z^l$-stability, by moving the polarisation towards the fiber direction in the ample cone while keeping the volume of the polarisation fixed. We describe conditions under which a slope stable torsion-free sheaf is taken by a Fourier-Mukai transform to a $Z^l$-stable object, and describe a modification upon which a $Z^l$-semistable object is taken by the inverse Fourier-Mukai transform to a slope semistable torsion-free sheaf.  We also study wall-crossing for Bridgeland stability, and show that 1-dimensional twisted Gieseker semistable sheaves are taken by a Fourier-Mukai transform to Bridgeland semistable objects.

\end{abstract}

\maketitle
\tableofcontents

\section{Introduction}

The problem of whether stable sheaves remain stable under a Fourier-Mukai transform has a long history, and extensive literature has been devoted to it.  We point to \cite{FMNT} as a comprehensive overview of results on this topic. On elliptic surfaces alone, this problem has been studied for a variety of motivations including  the construction of stable sheaves with prescribed Chern classes \cite{YosAS,FMW}, birational properties of moduli of sheaves \cite{FMTes,bernardara2014},  the moduli of instantons in gauge theory \cite{jardim2003fourier}, and strange duality \cite{marian2013generic}, just to illustrate the breadth of works among the large amount of literature.


In this article, we study the Fourier-Mukai transforms of stable sheaves on elliptic surfaces using a fresh approach. Our key idea is to consider how slope stability itself transforms, without fixing Chern classes.  

Since slope stability for sheaves is used in constructing Bridgeland stability conditions, understanding the Fourier-Mukai transform of slope stability  is key to  understanding the Fourier-Mukai transform of Bridgeland stability.  When $X$ is a K3 surface, the action of the autoequivalence group  on the derived category  $D^b(X)$ of coherent sheaves on $X$ is intimately related to the geometry of the space of Bridgeland stability conditions on $D^b(X)$ \cite{SCK3}.  More generally, when $X$ is a smooth projective variety, solutions to certain equations involving this group action give Gepner-type stability conditions, which impose constraints on internal symmetries of Donaldson-Thomas type invariants \cite{toda2013gepner}, or give rise to  pseudo-Anosov autoequivalences, which are  related to the existence of stability conditions on the Fukaya category  \cite{dimitrov2013dynamical}.   In a subsequent article \cite{Lo20}, the second author builds on the techniques in this article, and describes explicitly the images of certain Bridgeland stability conditions under Fourier-Mukai transforms on elliptic surfaces.


Recall that Bridgeland's construction of geometric stability conditions depends on the choice of a polarisation $\omega$. More precisely, the map
$$
Z_{\omega}=-\ch_2+\frac{\omega^2}{2}\ch_0+i\omega\ch_1
$$
is the central charge of a stability condition on a heart $\mathcal{B}_{\omega}$, which is obtained as a tilt of $\Coh(X)$ using the Mumford slope
$$
\mu_{\omega}=\frac{\omega\ch_1}{\ch_0}.
$$
A Weierstra{\ss} elliptic surface $p\colon X\rightarrow B$ comes endowed with a non-trivial Fourier-Mukai autoequivalence $\Phi\colon D^b(X)\rightarrow D^b(X)$, whose kernel is the relative Poincar\'e sheaf for the fibration $p$, i.e., the universal sheaf for the moduli problem of parametrizing degree-zero, rank-one torsion-free sheaves on the fibers of $p$. Since the Picard rank of $X$ is at least two, we can vary the polarisation in the ample cone and aim to find a stability condition $(Z_{\omega}, \mathcal{B}_{\omega})$, for which the Fourier-Mukai transform of a slope stable sheaf is stable.   This turns out to be not exactly the case and we will rather construct a polynomial stability condition satisfying this requirement.  Roughly speaking, this polynomial stability is obtained by moving the ample class towards the fiber direction in the ample cone of $X$ while fixing the volume of $\omega$.  Our main results can be summarized as follows:

\begin{thm}[Theorem \ref{thm:Lo14Thm5-analogue}, Theorem \ref{thm:main1}, and Theorem \ref{assymp_one_dim}]\label{main_in_intro}
Let $p\colon X\rightarrow B$ be a Weierstra{\ss} elliptic surface. Denote by $\Theta$ its canonical section and by $f$ the fiber class. Let $m>0$ such that $\Theta + m f$ is ample. Denote $e=-\Theta^2$. Fix $\alpha>0$ and let $\overline{\omega}=\alpha^{-1}(\Theta+mf)+f$. Consider the family of stability conditions $(Z_{\omega},\mathcal{B}_{\omega})$, where $\omega=u(\Theta+mf)+vf$ with $u,v$ on the curve
$$
\omega^2=2(\alpha+m-e).
$$
Then
\begin{enumerate}
\item $Z_{\omega}$ defines a polynomial stability condition $Z^l$ with parameter $v$ over a limit heart $\mathcal{B}^l$.
\item If $E$ is a $\mu_{\overline{\omega}}$-stable torsion-free sheaf with $2\alpha\ch_1(E)\cdot \overline{\omega}-e\ch_0(E)>0$, then $\Phi(E)[1]$ is $Z^l$-semistable.
\item If $E$ is a $\mu_{\overline{\omega}}$-stable locally free sheaf then $\Phi(E)[1]$ is $Z^l$-stable.
\item If $\mathcal{E}$ is a 1-dimensional twisted $\overline{\omega}$-Gieseker semistable sheaf with $\ch_1(\mathcal{E})\cdot f>0$ and $2\ch_2(\mathcal{E})-e\ch_1(\mathcal{E})\cdot f\geq 0$, then $\Phi(\mathcal{E})$ is $Z^l$-semistable for $\alpha+m\gg 0$.
\end{enumerate}
\end{thm}

We sometimes refer to the curve $\omega^2=2(\alpha+m-e)$ as a `volume section'.


In terms of the organisation of the article, after setting up the preliminaries and introducing the cohomological Fourier-Mukai transforms in Section \ref{sec:prelim}, we give the precise construction of $Z^l$-stability  on a Weierstra{\ss} surface in Section \ref{sec:limitBridgelandconstr}.  In Section \ref{sec:maintheorem}, we prove parts (2) and (3) of Theorem \ref{main_in_intro}, which compares slope stability and $Z^l$-stability for locally-free sheaves (Theorem \ref{thm:Lo14Thm5-analogue}). Section \ref{sec:HNproperty} is dedicated to the proof of the Harder-Narasimhan property for $Z^l$-stability, which concludes the proof of part (1) of Theorem \ref{main_in_intro}. In Section \ref{assymp_one_dim} we study the Fourier-Mukai transforms of 1-dimensional sheaves, part (4) of Theorem \ref{main_in_intro} is the content of Theorem \ref{assymp_one_dim}. At this point in the article, we begin fixing Chern characters and use the theory of $Z^l$-stability we have developed to study Fourier-Mukai transforms of stable sheaves.  This comes down to studying wall-crossing for Bridgeland stability conditions, and we give two approaches of different flavours.

The first approach is contained  in Section \ref{sec:trans1dimsh}, where  we consider walls given by Chern characters $\ch$ where $\ch_1$ is a positive multiple of the fiber class $f$ of the elliptic surface.  When the  elliptic surface has Picard rank two, we use Bogomolov inequalities to bound mini-walls on the curve along which $Z^l$-stability is defined.  This shows that the moduli space of Bridgeland stability at the far end of this curve coincides with the moduli space of  $Z^l$-stability.  As a result, we obtain Corollary \ref{cor:main1}, which says that if $\mathcal{E}$ is a 1-dimensional twisted Gieseker semistable sheaf, which has  positive twisted Euler characteristic and positive fiber degree $f\ch_1$, then its Fourier-Mukai transform is a Bridgeland stable object with 2-dimensional support.

The second approach is contained in Sections \ref{sec:BrivslimitBri} and \ref{sec:example}.  For this approach, we  begin by studying the asymptotics of walls in Section \ref{sec:BrivslimitBri}.  Then, in Section \ref{sec:example}, we apply the computations to elliptic surfaces of Picard rank two with a strictly negative  section.  Combined with Arcara-Miles' result on destabilising objects for line bundles, we obtain Proposition  \ref{prop:OxaLThetatransformss}, which roughly says that if $\mathcal{L}$ is  a line bundle of fiber degree at least 2, then it is a Bridgeland stable object, and  its inverse Fourier-Mukai transform is   a slope semistable locally free sheaf.

Proposition  \ref{prop:OxaLThetatransformss} is similar to a result due to the second author and Zhang on some  Weierstra{\ss} elliptic threefolds \cite[Theorem 4.4]{LZ3}, which says that if $\mathcal{L}$ is a line bundle of nonzero fiber degree, then its  Fourier-Mukai transform  is a slope stable  locally free sheaf. The argument for this threefold result, however, does not appear to reduce directly to the surface case.  Previous results on the transforms of stable torsion-free sheaves on elliptic threefolds include the works of C{\u{a}}ld{\u{a}}raru \cite{caldararu2000derived,cualduararu2002fiberwise}.

The essential ideas in Sections \ref{sec:limitBridgelandconstr} through \ref{sec:HNproperty} have also  appeared in the second author's preceding works on  a product elliptic threefold \cite{Lo14} and Weirstra{\ss} elliptic threefolds over a Fano or numerically $K$-trivial base  \cite{Lo15}.


\subsection*{Acknowledgements} The first author was supported by IBS-R003-D1. He thanks  California State University, Northridge for their support and hospitality during his visits in November 2017 and November 2018. He thanks Ludmil Katzarkov for comments and National Research University Higher School of Economics for the hospitality during the conference ``Who is Who in mirror symmetry". He also thanks Jihun Park and Tobias Ekholm for support. The second author  would  like to thank the National Center for Theoretical Sciences in Taipei for their support and hospitality during  December 2016-January 2017, when this work was initiated.  He also  thanks the Center for Geometry and Physics, Institute for Basic Science in Pohang, South Korea, for their support and hospitality throughout his visits in June-July 2017 and January 2018 during which much of this work was completed. The third author is supported by the FAPESP grant number 2020/06938-4, which is part of the FAPESP Thematic Project 2018/21391-1. He also would like to thank the Department of Mathematics at Universidad de los Andes for the excellent working conditions during his time as a postdoctoral assistant when part of this work was completed. The authors also thank Daniele Arcara, Ching-Jui Lai, Mu-Lin Li and Ziyu Zhang for answering their questions on surfaces, and the referees for suggestions on how to improve the manuscript.

\section{Preliminaries}\label{sec:prelim}

\paragraph[Our elliptic fibration]  Throughout \label{para:ourellipticfibration} this article, unless otherwise stated, we will write $p : X \to B$ to denote an elliptic surface that is a Weierstra{\ss} fibration in the sense of \cite{FMNT}  and \cite[Definition (II.3.2)]{MirLec}.  We do not place any restriction on the Picard rank of $X$ until the second half of  the paper.

\subparagraph[Elliptic surface] By an elliptic surface $p : X \to B$, we mean a flat morphism where $X$ is a smooth projective surface and $B$ is a smooth projective curve, and all the fibers are Gorenstein curves of arithmetic genus 1 \cite[Definition 6.8]{FMNT}.

\subparagraph[Weierstra{\ss} elliptic surface]  By a Weierstra{\ss} elliptic surface, or simply a Weierstra{\ss} surface, we will mean an elliptic surface $p : X \to B$ that is also a Weierstra{\ss} fibration in the sense of \cite[Definition 6.10]{FMNT}, which requires that all the fibers of $p$ are geometrically integral, and that $p$ admits a section $\sigma : B \to X$ whose image $\Theta = \sigma (B)$ does not meet any singular point of any fiber.

\paragraph[The N\'{e}ron-Severi group $\mathrm{NS}(X)$] Since \label{para:NSX} our elliptic fibration $p$ is assumed to be Weierstra{\ss}, there exists a section, and hence  the Picard rank of $X$ is finite by the Shioda-Tate formula \cite[VII 2.4]{MirLec}, while the N\'{e}ron-Severi group $\mathrm{NS}(X)$ is generated by the fiber class $f$ and a finite number of sections $\Theta_0 := \Theta, \Theta_1, \cdots, \Theta_r$ for some $r \geq 0$ \cite[VII 2.1]{MirLec}.

\paragraph[Geometry of $X$] The \label{para:geometry of X} fundamental line bundle of $p : X \to B$ is defined to be the line bundle  $\mathbb{L}:=(R^1p_*\OO_X)^{-1}\simeq p_* \omega_{X/B}$ \cite[II 3.6]{MirLec}. We also set   $\bar{K}:=c_1(p_* \omega_{X/B})\simeq c_1(\mathbb{L})$ and write  $e:=\degree(\mathbb{L})$. Then $p^\ast \bar{K}\equiv ef$ while we also know  that $\deg \mathbb{L} = - \Theta^2$ \cite[Theorem 7.20]{Fri98}. Hence  $\Theta^2=-e.$ Then by \cite[Proposition (III 1.1)]{MirLec} and \cite[(6.13)]{FMNT},
\begin{equation} \label{eq: canonical sheaf}
    \omega_X \simeq p^*(\omega_B\otimes \mathbb{L})\simeq p^*(\omega_B)\otimes \omega_{X/B}.
\end{equation}
By the  adjunction formula, we have $\Theta.(\Theta+K_X)= 2 g(B)-2=\Theta_i.(\Theta_i+K_X)$ and hence  $K_X \equiv (2g(B)-2+e)f$ and $\Theta_i^2=\Theta^2=-e$.

A classification of  Weierstra{\ss} elliptic surfaces  is contained in  \cite[Lemma (III 4.6)]{MirLec}.

\paragraph[Notation] We \label{para:notation} collect here preliminary notions and notations that will be used throughout the article.

\subparagraph[Twisted Chern character] For any divisor $B$ on a smooth projective  surface $X$ and any $E \in D^b(X)$, the twisted Chern character $\ch^B(E)$ is defined as
\[
\ch^B(E) = e^{-B}\ch(E)
= (1-B+\tfrac{B^2}{2})\ch(E).
\]
  We write $\ch^B(E) = \sum_{i=0}^2 \ch_i^B(E)$ where
\begin{align*}
  \ch_0^B(E) &= \ch_0(E), \\
  \ch_1^B(E) &= \ch_1(E)-B\ch_0(E), \\
  \ch_2^B(E) &= \ch_2(E) -B\ch_1(E)+\tfrac{B^2}{2}\ch_0(E).
\end{align*}
We sometimes refer to the divisor $B$ involved in the twisting of the Chern character as the `$B$-field'.  In this article, there should be no risk of confusion as to whether $B$ refers to the base of the elliptic fibration $p$ or a $B$-field.

\subparagraph[Cohomology] Suppose $\Ac$ is an abelian category and $\Bc$ is the heart of a t-structure on $D^b(\Ac)$.  For any object $E \in D^b(\Ac)$, we will write $\mathcal{H}^i_{\Bc}(E)$ to denote the $i$-th cohomology object of $E$ with respect to the t-structure with heart $\Bc$.  When $\Bc = \Ac$, i.e.\ when the aforementioned t-structure is the standard t-structure on $D^b(\Ac)$, we will write $H^i(E)$ instead of $\mathcal{H}^i_{\Ac}(E)$.

Given a smooth projective variety $X$, the dimension of an object $E \in D^b(X)$ will be denoted by $\dimension E$, and refers to the dimension of its support, i.e.\
\[
  \dimension E = \dimension \bigcup_i \mathrm{supp}\, H^i(E).
\]
For a coherent sheaf $E$, we have $\dimension E = \dimension \mathrm{supp} (E)$.

\subparagraph[Torsion pairs and tilting]  A \label{para:torsionpairtilting} torsion pair $(\Tc, \Fc)$ in an abelian category $\mathcal{A}$ is a pair of full subcategories $\Tc, \Fc$ such that
\begin{itemize}
\item[(i)] $\Hom_{\Ac}(E', E'')=0$ for all $E' \in \Tc, E'' \in \Fc$.
\item[(ii)] Every object $E \in \Ac$ fits in an $\Ac$-short exact sequence
\[
0 \to E' \to E \to E'' \to 0
\]
for some $E' \in \Tc, E'' \in \Fc$.
\end{itemize}
The decomposition of $E$ in (ii) is canonical \cite[Chapter 1]{HRS}, and we will  refer to it as the $(\Tc,\Fc)$-decomposition of $E$ in $\Ac$. Whenever we have a torsion pair $(\Tc, \Fc)$ in an abelian category $\mathcal{A}$, we will refer to $\Tc$ (resp.\ $\Fc$) as the torsion class (resp.\ torsion-free class) of the torsion pair.  The extension closure in $D^b(\Ac)$
\[
  \Ac' = \langle \Fc [1], \Tc \rangle
\]
is the heart of a t-structure on $D^b(\Ac)$ and hence an abelian subcategory of $D^b(\Ac)$.  We call $\Ac'$ the tilt of $\Ac$ at the torsion pair $(\Tc, \Fc)$.  More specifically, the category $\Ac'$ is the heart of the t-structure $(D^{\leq 0}_{\Ac'}, D^{\geq 0}_{\Ac'})$ on $D^b(\Ac)$ where
\begin{align*}
  D^{\leq 0}_{\Ac'} &= \{ E \in D^b(\Ac) : \mathcal{H}_{\Ac}^0 (E)\in \Tc, \mathcal{H}^i_{\Ac} (E)= 0 \, \forall\, i > 0 \}, \\
  D^{\geq 0}_{\Ac'} &= \{ E \in D^b(\Ac) : \mathcal{H}_{\Ac}^{-1} (E)\in \Fc, \mathcal{H}^i_{\Ac} (E)= 0 \,\forall\, i < -1 \}.
\end{align*}
A subcategory of $\Ac$ will be called a torsion class (resp.\ torsion-free class) if it is the torsion class (resp.\ torsion-free class) in some torsion pair in $\Ac$.  By a lemma of Polishchuk \cite[Lemma 1.1.3]{Pol}, if $\Ac$ is a noetherian abelian category, then every subcategory that is closed under extension and quotient in $\Ac$ is a torsion class in $\Ac$.

For any subcategory $\mathcal{C}$ of an abelian category $\mathcal{A}$, we will set
\[
  \mathcal{C}^\circ = \{ E \in \mathcal{A} : \Hom_{\mathcal{A}}(F,E)=0 \text{ for all } F \in \mathcal{C} \}
\]
when $\mathcal{A}$ is clear from the context.  Note that whenever $\mathcal{A}$ is noetherian and $\mathcal{C}$ is closed under extension and quotient in $\mathcal{A}$, the pair $(\mathcal{C},\mathcal{C}^\circ)$ gives a torsion pair in $\mathcal{A}$.


\subparagraph[Torsion $n$-tuples] A torsion $n$-tuple $(\Cc_1, \Cc_2,\cdots,\Cc_n)$ in an abelian category $\Ac$ as defined in \cite[Section 2.2]{Pol2} is a collection of full subcategories of $\Ac$ such that
\begin{itemize}
\item $\Hom_{\Ac} (C_i,C_j)=0$ for any $C_i \in \mathcal C_i, C_j \in \mathcal C_j$ where $i<j$.
\item Every object $E$ of $\Ac$ admits a filtration in $\Ac$
\[
  0=E_0 \subseteq E_1 \subseteq E_2 \subseteq \cdots \subseteq E_n = E
\]
where $E_i/E_{i-1} \in \mathcal C_i$ for each $1 \leq i \leq n$.
\end{itemize}
(See also \cite[Definition 3.5]{Toda2}.)   Given a torsion $n$-tuple in $\Ac$ as above, the pair
\[
(\langle \Cc_1, \cdots, \Cc_i\rangle, \langle \Cc_{i+1}, \cdots, \Cc_n \rangle )
\]
 is a torsion pair in $\Ac$ for any $1 \leq i \leq n-1$.

\subparagraph[Fourier-Mukai transforms] For any Weierstra{\ss} elliptic fibration $p : X \to B$ in the sense of \cite[Section 6.2]{FMNT}, there is a pair of relative Fourier-Mukai transforms $\Phi, \whPhi : D^b(X) \overset{\thicksim}{\to} D^b(X)$ whose kernels are both sheaves on $X \times _B X$, satisfying
\begin{equation}\label{eq:PhiwhPhiisidshifted}
  \whPhi \Phi = \mathrm{id}_{D^b(X)}[-1] = \Phi \whPhi.
\end{equation}
In particular, the kernel of $\Phi$ is the relative Poincar\'{e} sheaf for the fibration $p$, which is a universal sheaf for the moduli problem that parametrises degree-zero, rank-one torsion-free sheaves on the fibers of $p$.  An object $E \in D^b(X)$ is said to be $\Phi$-WIT$_i$ if $\Phi E$ is a coherent sheaf sitting at degree $i$.  In this case, we write $\wh{E}$ to denote a coherent sheaf satisfying  $\Phi E \cong \wh{E} [-i]$ up to isomorphism.  The notion of $\whPhi$-WIT$_i$ can similarly be defined.  The identities \eqref{eq:PhiwhPhiisidshifted} imply that, if a coherent sheaf $E$ on $X$ is $\Phi$-WIT$_i$ for $i=0, 1$, then $\wh{E}$ is $\whPhi$-WIT$_{1-i}$.  For $i=0,1$, we will define the category
\[
  W_{i,\Phi} = \{ E \in \Coh (X) : E \text{ is $\Phi$-WIT$_i$} \}
\]
and similarly for $\whPhi$.  Due to the symmetry between $\Phi$ and $\whPhi$, the properties held by $\Phi$ also hold for $\whPhi$.     See \cite[Section 6.2]{FMNT} for more background on the functors $\Phi, \whPhi$.

\subparagraph[Subcategories of $\Coh (X)$]  Let $p : X \to B$ be an elliptic surface as in \ref{para:ourellipticfibration}. For any integers $d \geq e$, we set
\begin{align*}
\Coh^{\leq d}(X) &= \{ E \in \Coh (X): \dimension \mathrm{supp}(E) \leq d\} \\
\Coh^d(p)_e &= \{ E \in \Coh (X): \dimension \mathrm{supp}(E) = d, \dimension p (\mathrm{supp}(E)) = e\} \\
\{ \Coh^{\leq 0} \}^\uparrow &= \{ E \in \Coh (X): E|_b \in \Coh^{\leq 0} (X_b) \text{ for all closed points $b \in B$} \}
\end{align*}
where $\Coh^{\leq 0}(X_b)$ is the category of coherent sheaves supported in dimension 0 on the fiber $p^{-1}(b)=X_b$, for the closed point $b \in B$.  We will refer to coherent sheaves that are supported on a finite number of fibers of $p$ as fiber sheaves.  Adopting the notation in \cite[Section 3]{Lo14}\footnote{\label{footnote-box-notation} When the Picard rank of the surface is two, these box notations correspond exactly to the signs of Chern characters of objects in these categories.}, we also define
\begin{align*}
   \scalea{\gyoung(;;,;;+)}&:= \Coh^{\leq 0}(X) \\
    \scalea{\gyoung(;;+,;;+)} &:= \{ E \in \Coh^1(\pi)_0 :  \text{ all $\mu$-HN factors of $E$ have $\infty>\mu>0$}\} \\
  \scalea{\gyoung(;;+,;;0)} &:=\{ E \in \Coh^1(\pi)_0 : \text{ all $\mu$-HN factors of $E$ have $\mu=0$}\} \\
  \scalea{\gyoung(;;+,;;-)} &:= \{ E \in \Coh^1(\pi)_0 :  \text{ all $\mu$-HN factors of $E$ have $\mu<0$}\} \\
   \scalea{\gyoung(;;*,;+;*)} &:= \Coh^1(\pi)_1 \cap \{\Coh^{\leq 0}\}^\uparrow\\
\scalea{\gyoung(;+;*,;+;*)} &= \{ E \in W_{0,\whPhi}: \dimension E = 2\} \\
\scalea{\gyoung(;+;*,;0;*)} &= \{ E \in \Phi (\{\Coh^{\leq 0}\}^\uparrow \cap \Coh^{\leq 1}(X)) : \dimension E = 2\} \\
\scalea{\gyoung(;+;*,;-;*)} &= \{ E \in  W_{1,\wh{\Phi}} : \dimension E =2, f\ch_1(E)\neq 0 \} .
\end{align*}
Note that the definitions of $\scalea{\gyoung(;+;*,;+;*)}, \scalea{\gyoung(;+;*,;0;*)}$ and $\scalea{\gyoung(;+;*,;-;*)}$ depend on the Fourier-Mukai functor $\whPhi$.  We will use the same notation to denote the corresponding category defined using $\whPhi$; it will always be clear from the context which Fourier-Mukai functor the definition is with respect to.  The Fourier-Mukai transform $\Phi$ induces the following equivalences, as already observed in \cite[Remark 3.1]{Lo14}:
\[
  \xymatrix @-1.3pc{
\scalea{\gyoung(;;,;;+)} \ar[dr] & \scalea{\gyoung(;;+,;;+)} \ar@/^1.5pc/[dd] & \scalea{\gyoung(;;*,;+;*)} \ar[dr] & \scalea{\gyoung(;+;*,;+;*)}  \ar@/^1.5pc/[dd]  \\
 & \scalea{\gyoung(;;+,;;0)} & & \scalea{\gyoung(;+;*,;0;*)} \\
  & \scalea{\gyoung(;;+,;;-)} & & \scalea{\gyoung(;+;*,;-;*)}
 }
\]
A concatenation of more than one such diagram will mean the extension closure of the categories involved; for example, the concatenation
\[
\xymatrix @-2.3pc{
\scalea{\gyoung(;;,;;+)} & \scalea{\gyoung(;;+,;;+)} \\
& \scalea{\gyoung(;;+,;;0)}
}
\]
is the extension closure of all slope semistable fiber sheaves of slope at least zero (including sheaves supported in dimension zero, which are slope semistable fiber sheaves of slope $+\infty$).

The category $\Coh^{\leq d}(X)$ for any integer $d \geq 0$, as well as $\{\Coh^{\leq 0}\}^\uparrow$ and $W_{0,\whPhi}$ are all torsion classes in $\Coh (X)$.  From \ref{para:torsionpairtilting}, each of these torsion classes determines a tilt of $\Coh (X)$, and hence determines a t-structure on $D^b(X)$. For instance, we have the torsion pairs $(W_{0,\whPhi}, W_{1,\whPhi})$ and $(\Coh^{\leq d}(X), \Coh^{\geq d+1}(X))$ in $\Coh (X)$.

\subparagraph[Slope functions]  Suppose \label{para:slopelikefunctions} $\Ac$ is an abelian category.  Any  function $\mu$ on $\Ac$ of the following form will be referred to as a slope  function 
\[
  \mu (F) = \begin{cases} \frac{C_1(F)}{C_0(F)} &\text{ if $C_0(F) \neq 0$} \\
  +\infty &\text{ if $C_0(F)=0$} \end{cases}
\]
where $C_0, C_1 : K(\Ac) \to \mathbb{Z}$ are a pair of group homomorphisms  satisfying: (i) $C_0(F) \geq 0 $ for any $F \in \Ac$; (ii) if $F \in \Ac$ satisfies $C_0(F)=0$, then $C_1 (F) \geq 0$.  The additive group $\mathbb{Z}$ in the definition of $\mu$ can be replaced by any  discrete additive subgroup of $\mathbb{R}$.  Whenever $\Ac$ is a noetherian abelian category, every slope function $\mu$ possesses the Harder-Narasimhan (HN) property \cite[Section 3.2]{LZ2}; we will then say an object $F \in \Ac$ is $\mu$-stable (resp.\ $\mu$-semistable) if, for every short exact sequence $0 \to M \to F \to N \to 0$ in $\Ac$ where $M,N \neq 0$, we have $\mu (M) < (\text{resp.} \leq \, )\, \mu (N)$.

\subparagraph[Slope stability] Suppose \label{para:muomegaBslopefctndefn} $X$ is a smooth projective surface with a fixed ample divisor   $\omega$ and a fixed divisor $B$.  For any coherent sheaf $E$ on $X$, we  define
\[
  \mu_{\omega,B} (E) = \begin{cases}
  \frac{\omega \ch_1^B (E)}{\ch_0^B(E)} &\text{ if $\ch_0^B(E) \neq 0$} \\
  +\infty &\text{ if $\ch_0^B(E)=0$}\end{cases}.
\]
A coherent sheaf $E$ on $X$ is said to be $\mu_{\omega,B}$-stable or slope stable (resp.\ $\mu_{\omega,B}$-semistable or slope semistable) if, for every short exact sequence in $\Coh (X)$ of the form
\[
0 \to M \to E \to N \to 0
\]
where $M,N \neq 0$, we have $\mu_{\omega,B} (M) < (\text{resp.}\, \leq) \, \mu_{\omega,B} (N)$.  Note that for any  coherent sheaf $M$ on $X$ with $\ch_0(M) \neq 0$,  we have
\[
\mu_{\omega,B} (M) = \frac{\omega \ch_1^B(M)}{\ch_0(M)} = \frac{\omega\ch_1(M)-\omega B\ch_0(M)}{\ch_0(M)} = \mu_\omega (M) -\omega B.
\]
Hence $\mu_{\omega,B}$-stability is equivalent to $\mu_\omega$-stability for  coherent sheaves.  When $B=0$, we often write $\mu_\omega$ for $\mu_{\omega,B}$.  Also, since $\mu_\omega$-stability has the HN property whenever $\omega$ is an $\mathbb{R}$-divisor that is a movable class \cite[Corollary 2.27]{greb2016movable}, the slope function $\mu_{\omega, B}$ also has the HN property for any $\mathbb{R}$-divisors $\omega, B$ where $\omega$ is ample.

\subparagraph[Bridgeland stability conditions on surfaces] Suppose \label{para:Bridgestabonsurfaces} $X$ is a smooth projective surface.  For any ample divisor $\omega$ and another divisor $B$ on $X$, we can define the following subcategories of $\Coh (X)$
\begin{align*}
  \Tc_{\omega,B} &= \langle F \in \Coh (X) :  F \text{ is $\mu_{\omega,B}$-semistable}, \, \mu_{\omega, B} (F) > 0\rangle, \\
  \Fc_{\omega, B} &= \langle F \in \Coh (X) : F \text{ is $\mu_{\omega, B}$-semistable}, \, \mu_{\omega,B} (F) \leq 0 \rangle.
\end{align*}
Since the slope function $\mu_{\omega,B}$ has the Harder-Narasimhan property, the pair $(\Tc_{\omega,B},\Fc_{\omega,B})$ is a torsion pair in $\Coh (X)$.  The extension closure
\[
  \Bc_{\omega,B} = \langle \Fc_{\omega,B} [1], \Tc_{\omega,B} \rangle
\]
in $D^b(X)$ is thus  a tilt of the heart $\Coh (X)$, i.e.\ $\Bc_{\omega,B}$ is  the heart of a bounded t-structure on $D^b(X)$ and is an abelian subcategory of $D^b(X)$.   If we set
\begin{equation}\label{eq:Zwformula}
  Z_{\omega,B} (F) = -\int_X e^{-i\omega}\ch^B(F) = -\ch_2^B(F) + \tfrac{\omega^2}{2}\ch_0(F) + i\omega \ch_1^B(F),
\end{equation}
then the pair 
\begin{equation}
    \label{eq:Bridgeland stab}
    (\Bc_{\omega,B}, Z_{\omega,B})=:\sigma_{\omega,B}
\end{equation}
gives a Bridgeland stability condition on $D^b(X)$, as shown by Arcara-Bertram  in \cite{ABL}.   In particular, for any nonzero object $F \in \Bc_{\omega,B}$, the complex number $Z_{\omega,B} (F)$ lies in the upper-half complex plane (that includes the negative real axis)
\[
  \mathbb{H} = \{ re^{i\pi \phi} : r>0, \phi \in (0,1] \}.
\]
This allows us to define the phase $\phi (F)$ of any nonzero object $F \in \Bc_{\omega,B}$ using the relation
\[
  Z_{\omega,B} (F) \in \mathbb{R}_{>0} e^{i \pi \phi (F)}  \text{\quad where } \phi (F) \in (0,1].
\]
We then say an object $F \in \Bc_{\omega,B}$ is $Z_{\omega,B}$-stable (resp.\ $Z_{\omega,B}$-semistable) if, for all $\Bc_{\omega,B}$-short exact sequence
\[
0 \to M \to F \to N \to 0
\]
where $M, N \neq 0$, we have $\phi (M) < \phi (N)$ (resp.\ $\phi (M) \leq \phi (N)$).

If $B=0$, we write $Z_\omega$ and $\Bc_\omega$ instead of $Z_{\omega,0}$ and $\Bc_{\omega,0}$ respectively.

\paragraph[The cohomological Fourier-Mukai transforms] For  \label{para:cohomFMTformulas} any $E \in D^b(X)$, let \footnote{\label{footnote-different-notations} We have used two different system of notations for components of Chern characters. When we compute Fourier-Mukai transform, we follow the notations in \cite[Section 6.2]{FMNT}. When we compute wall-crossing formulas, we follow the notations in \cite{Mac2}. See \eqref{eq-ch-Maciocia} in Appendix~\ref{sec:Bridgeland wall-chamber structures}.}
\begin{align}
  n &= \ch_0(E), \notag\\
  d = f\ch_1(E)&, \,\,  \text{\qquad } c = \Theta\ch_1(E), \notag\\
  s &= \ch_2(E). \label{eq:chEnotation}
\end{align}
Then from the cohomological Fourier-Mukai transform in   \cite[(6.21)]{FMNT} we have
\begin{align}
\ch_0 (\Phi E) &= d, \notag\\
\ch_1 (\Phi E) &= -\ch_1(E) + dp^\ast \bar{K} + (d-n)\Theta + (c-\tfrac{1}{2}ed + s)f, \notag\\
\ch_2 (\Phi E) &= (-c-de + \tfrac{1}{2}ne) \label{eq:cohomFMT-Phi}
\end{align}
where $\Theta^2 = -e$ and $\bar{K} = c_1 (p_\ast \omega_{X/B})$ as in \ref{para:geometry of X}.  Since $p^\ast \bar{K} \equiv ef$, we have $ \ch_1(\Phi E). f =  -n$ and $\ch_1(\Phi E). \Theta =  (s-\tfrac{e}{2}d)+ne$.
In particular, for any $m \in \mathbb{R}$ we have
\begin{align}
\ch_1(\Phi E). f &=  -n,\notag\\
\ch_1(\Phi E). (\Theta + mf) &= s-\tfrac{e}{2}d + (e-m)n. \label{eq:cohomFMT-Phi-ch1}
\end{align}  On the other hand, from \cite[(6.22)]{FMNT} we have
\begin{align}
  \ch_0 (\wh{\Phi}E) &= d, \notag\\
  \ch_1 (\wh{\Phi}E) &= \ch_1(E) - np^\ast \bar{K} - (d+n)\Theta + (s+en-c-\frac{e}{2}d)f, \notag \\
  \ch_2 (\wh{\Phi}E) &= -(c+de+\tfrac{e}{2}n). \label{eq::cohomFMT-hatPhi}
\end{align}
This gives $\ch_1 (\wh{\Phi} E). f = -n$ and $\ch_1 (\wh{\Phi}E). \Theta = s+\tfrac{e}{2}d + ne$.  In particular, for any $m \in \mathbb{R}$ we have
 \begin{align}
  \ch_1 (\wh{\Phi} E). f &= -n, \notag\\
  \ch_1 (\wh{\Phi} E). (\Theta + mf) &= s+\tfrac{e}{2}d+(e-m)n.  \label{eq::cohomFMT-hatPhi-ch1}
\end{align}

\paragraph[Some intersection numbers] Here \label{para:basicintersectionnumbers} we collect some intersection numbers that will be used throughout the rest of the paper.  For any $m \in \RR$ we have
\begin{equation*}
  (\Theta +mf)^2 = \Theta^2 + 2m = 2m-e.
\end{equation*}
Recall that for  any section $\Theta$ of the fibration $p$, the divisor $\Theta + mf$ on $X$ is  ample for $m \gg 0$ \cite[Proposition 1.45]{KM}. Let us fix an $m$ such that $\Theta + mf$ is ample. We will often work with a polarisation of the form
\begin{equation}\label{eq:omeganotation}
  \omega = u(\Theta + mf) + vf
\end{equation}
for $u,v>0$, which gives
\begin{equation} \label{eq:volume_omega_half}
  \tfrac{\omega^2}{2} = (m-\tfrac{e}{2})u^2 + uv.
\end{equation}
If we use the notation for $\ch(E)$ in \eqref{eq:chEnotation} then $(\Theta + mf) \ch_1 (E) = c + md$ and
\begin{align*}
\omega \ch_1 (E) &=  (u(\Theta + mf)+vf)\ch_1(E) \\
&= uc + (um+v)d.
\end{align*}
If we also set
\[
 \oo = a(\Theta + mf)+bf,
\]
 where $a, b \in \RR$ and fix $B = \tfrac{e}{2}f$ then
\begin{align*}
  \oo \ch_1^B(E) &=  \oo (\ch_1(E) - \tfrac{e}{2}f\ch_0(E)) \\
  &= a(c-\tfrac{e}{2}n) + (am+b)d.
\end{align*}
Thus when $\oo$ is an ample divisor on $X$, we can write the twisted slope function $\mu_{\oo,B}$ as
\begin{equation}\label{eq:muBbbE}
  \mu_{\oo,B}(E) = \tfrac{1}{n} (a(c-\tfrac{e}{2}n) + (am+b)d).
\end{equation}
On the other hand, when $\omega$ is an ample divisor on $X$, with respect to the central charge \eqref{eq:Zwformula} and using \eqref{eq:cohomFMT-Phi-ch1} we have
\begin{align}
  Z_\omega (\Phi E [1]) &=  \ch_2(\Phi E)  -\tfrac{\omega^2}{2}\ch_0(\Phi E) - i\omega \ch_1(\Phi E) \notag\\
  &= (-c-de+\tfrac{e}{2}n) - ((m-\tfrac{e}{2})u^2+uv)d - i \left( u ( s-\tfrac{e}{2}d+(e-m)n) - vn \right)  \notag\\
  &= (-c + \tfrac{e}{2}n) - ( (m-\tfrac{e}{2})u^2 + uv +e)d + i \left( u (-( s-\tfrac{e}{2}d)+(m-e)n)+vn\right). \label{eq:ZwPhiEshift}
\end{align}

\paragraph[Heuristics and a volume section] Comparing \label{para:heuristics} the coefficients of the characteristic classes $(c-\tfrac{e}{2}n)$ and $d$ in the  expressions for $\mu_{\oo,B} (E)$ and $Z_\omega (\Phi E [1])$, we see that for fixed $m, a, b>0$, if $v \to \infty$ along the curve
\[
  \frac{am+b}{a} = (m-\tfrac{e}{2})u^2+uv+e,
\]
i.e.\
\begin{equation*}
  m + \tfrac{b}{a} = (m-\tfrac{e}{2})u^2 + uv +e,
\end{equation*}
then $\oo \ch_1^B(E)$ is a negative scalar multiple of $\Re Z_\omega (\Phi E [1])$, while $\Im Z_\omega (\Phi E [1])$ is dominated by a positive scalar multiple of $\ch_0(E)$.  This suggests that  for $v \gg 0$, $\mu_{\oo,B}$-stability for $E$ should be an `approximation' of  $Z_\omega$-stability up to the Fourier-Mukai transform $\Phi$, or that $Z_\omega$-stability is a `refinement' of $\mu_{\oo,B}$-stability for $E$ up to $\Phi$.  We will make this idea precise in Sections \ref{sec:limitBridgelandconstr} through \ref{sec:HNproperty}.  The computation above also motivates us to consider the change of variables
\begin{equation*}
  \beta = b, \,\,\alpha = \tfrac{b}{a},
\end{equation*}
so that $\oo$ can be written as
\begin{equation}\label{eq:oonotation}
  \oo = \tfrac{\beta}{\alpha} (\Theta + mf) + \beta f.
\end{equation}
Moreover, the $\mu_{\oo,B}$-stability depends only on $\alpha$ but not $\beta$. 
We can think of $\mu_{\oo,B}$-stability as being approximated by $Z_\omega$-stability as $v \to \infty$ along the curve
\begin{equation}\label{eq:vhyperbolaequation}
(m-\tfrac{e}{2})u^2 + uv =   \alpha+m - e,
\end{equation}\label{eq:vhyperbolaequation2}
which, by \eqref{eq:volume_omega_half}, is equivalent to 
\begin{equation} \label{eq:volume_omega}
    \omega^2= 2( \alpha+m - e).
\end{equation}
That is, by imposing the constraint \eqref{eq:vhyperbolaequation} we are fixing the volume of the polarisation $\omega$ while moving $\omega$ towards the fiber direction $f$ in the ample cone of $X$.  As a result, we refer to the plane curve \eqref{eq:vhyperbolaequation} as a \emph{volume section}.

In particular, the volume section \eqref{eq:vhyperbolaequation} is asymptotic to the curve
\begin{equation}\label{eq:vhyperbolaequation-asymptotic}
u=\frac{1}{v}\left(\alpha+m-e\right) \text{ as } v \rightarrow \infty.
\end{equation}
We will revisit \eqref{eq:vhyperbolaequation} in a new coordinate in \eqref{eq:curve-in-lambda-q-plane}. See also Figure~\ref{fig:q-lambda}.

\begin{figure}
\centering
\begin{tikzpicture}[thick]
        \begin{axis}[
            xmin=0, xmax=40, 
            ymin=0, ymax=1,
            axis x line = center, 
            axis y line = center,
            xtick = \empty,
            ytick = \empty,
            xlabel = {$v$},
            ylabel = {$u$},
            xlabel style={below right},
            ylabel style={above left},
            legend style = {nodes=right},
            legend pos = north east,
            clip mode = individual,
            axis line style=black!50,
            ]
            
            \node[below left] at (0,0) {$0$};
            
            \tikzmath{
            \m=3;
            \a=2;
            \e=2;
            \b=\m+ \a-\e;
            \c=-0.5*\e+\m;
            } 
  
            \addplot[black!30, samples=280, domain=2:30] {\b/x};
            \addplot[black!30, dotted, samples=100, domain=30:40] {\b/x};
            \addplot[black, samples=280, domain=2:30] {(sqrt(x*x+4*\b*\c)-x)/(2*\c)};
            \addplot[black, dotted, samples=100, domain=30:40] {(sqrt(x*x+4*\b*\c)-x)/(2*\c)};

        \end{axis}
\end{tikzpicture}
\caption{The volume section \eqref{eq:vhyperbolaequation} (black) and its asymptotic curve \eqref{eq:vhyperbolaequation-asymptotic} (gray) as $v\rightarrow\infty$.}
\label{fig:vhyperbolaequation}
\end{figure}
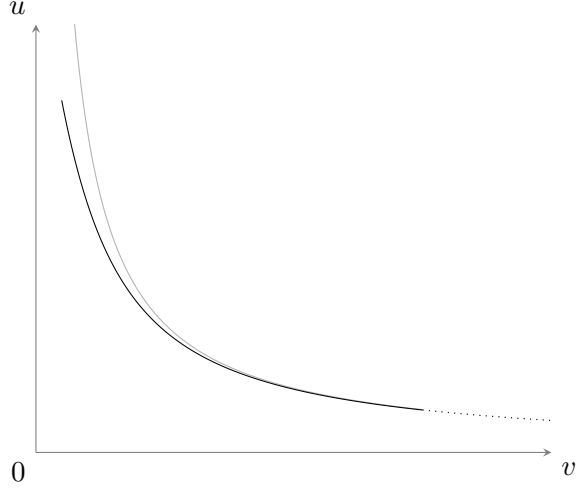

\paragraph[Decomposing $\mu_\omega$] Suppose $F$ is an object in $D^b(X)$.  With $\omega$ as in \eqref{eq:omeganotation},  we can rewrite $\mu_\omega (F)$ as
\begin{align}
  \mu_\omega (F) = \frac{\omega \ch_1(F)}{\ch_0(F)} &= u \frac{(\Theta + mf)\ch_1(F)}{\ch_0(F)}+ v \frac{f\ch_1(F)}{\ch_0(F)} \notag\\
  &= u \mu_{\Theta + mf}(F) + v \mu_f (F). \label{eq:muwdecomposition}
\end{align}
Recall that the divisor $\Theta + mf$ is ample on $X$ for $m \gg 0$ while $f$ is a nef divisor on $X$. Therefore,  both  $\mu_{\Theta + mf}$ and $\mu_f$ are slope functions with the Harder-Narasimhan property (see \ref{para:slopelikefunctions}).

\paragraph For \label{para:ch1realZwcomparison} fixed $\beta, \alpha >0$, with  $\oo$ as in \eqref{eq:oonotation}, $\omega$ as in \eqref{eq:omeganotation}, and $u,v>0$ under the constraint \eqref{eq:vhyperbolaequation}, we have the following observation that will be useful later on: with the same notation for $\ch(E)$ as in \ref{para:cohomFMTformulas}, for the $B$-field $B=\tfrac{e}{2}f$ we have
\begin{align}
  \oo \ch_1^B (E) &= \oo (\ch_1(E) - B \ch_0(E)) \notag\\
  &= \tfrac{\beta}{\alpha} ( (c-\tfrac{e}{2}n) + (m + \alpha)d) \notag\\
  &= - \tfrac{\beta}{\alpha} \Re Z_\omega (\Phi E [1]). \label{eq:ooch1twistReZePhiE1relation}
\end{align}
In particular, if $F$ is a $\whPhi$-WIT$_1$ sheaf on $X$ of nonzero rank  with $f\ch_1(F)=0$, then $\wh{F} = \whPhi F[1]$ is a sheaf supported in dimension 1, implying $\oo \ch_1^B(\wh{F}) = \oo \ch_1 (\wh{F})>0$.  Then
\[
  \Re Z_\omega (F) = \Re Z_\omega (\Phi \wh{F})=-\Re Z_\omega (\Phi \wh{F}[1])= \tfrac{\beta}{\alpha}\oo \ch_1 (\wh{F})>0.
\]

\begin{table}
\caption {A summary of notations for `before' and `after' the autoequivalence $\Phi [1]$. Let $m>0$ such that $\Theta+m f$ is ample. Fix $\alpha>0$. }
\begin{center}
\begin{tabular}{ c  c  c }
\hline
$E$ & $\underset{\whPhi}{\stackrel{\Phi[1]}{\rightleftharpoons}}$ & $F=\Phi E [1]$ \\
$B$-field $B=\frac{e}{2}f \equiv{p^\ast \bar{K}}/2$ &  & $B$-field $B=0$ \\
$\oo=\tfrac{\beta}{\alpha}(\Theta+mf)+ \beta f$  &  & $\omega=u(\Theta+mf)+vf$ \\
as \eqref{eq:oonotation} &    & $\omega=t H_\lambda$ via \eqref{eq-coordinate-change} \\
 & Volume section (\ref{eq:vhyperbolaequation}) or (\ref{eq:curve-in-lambda-q-plane}) & $Z_\omega=Z_{\omega,0}$ as (\ref{eq:Zwformula}) \\
$\mu_{\oo,B}$ as (\ref{para:muomegaBslopefctndefn}) & Limit along volume section  & $Z^l$ as (\ref{para:limit_Bridgeland_stability})\\
 & as $v\to \infty$ or $\lambda\to 0^+$. & \\
$\Coh(X)$ & & $\Phi(\Coh(X))[1]$\\
$\whPhi\Bc^l$ as (\ref{para:torsionfreeshtransf}) & & $\Bc^l$ as (\ref{para:TlFlproperties})\\
\hline
\end{tabular}
\end{center}
\end{table}

\section{Constructing a limit Bridgeland stability}\label{sec:limitBridgelandconstr}

Since the Bridgeland stability condition $(\Bc_\omega, Z_\omega)=\sigma_{\omega,0}$ \eqref{eq:Bridgeland stab} on $X$ depends on $\omega$, varying $\omega$ will change the stability condition accordingly (see \ref{para:Bridgestabonsurfaces}).  In this section, we will show that when $\omega$ is written in the form \eqref{eq:omeganotation}
\[
\omega = u(\Theta + mf) + vf
\]
with a fixed volume \eqref{eq:volume_omega} and $v \to \infty$ along to the  curve \eqref{eq:vhyperbolaequation}, we obtain a notion of stability with the Harder-Narasimhan property, which can be considered as a `limit Bridgeland stability'.

Due to the symmetry between $\Phi$ and $\whPhi$, all the results involving $\Phi$ and $\whPhi$ in this section and beyond still hold if we interchange $\Phi$ and $\whPhi$ (except for explicit computations involving Chern classes, since the cohomological Fourier-Mukai transforms corresponding to $\Phi$ and $\whPhi$ are different - see \ref{para:cohomFMTformulas}).

For the rest of this article, let us fix an $m >0$ so that $\Theta + mf$ is ample. We write $\omega$ in the form \eqref{eq:omeganotation} with $u,v>0$.


\begin{lem}\label{lem:TclFcldefinitions}
Suppose $u_0>0$ and $F \in \Coh (X)$.
\begin{itemize}
\item[(1)] The following are equivalent:
  \begin{itemize}
  \item[(a)] There exists $v_0>0$ such that $F \in \Fc_\omega$ for all $(v,u) \in (v_0,\infty) \times (0,u_0)$.
  \item[(b)] There exists $v_0>0$ such that, for every nonzero subsheaf $A \subseteq F$, we have $\mu_\omega (A) \leq 0$ for all $(v,u) \in (v_0,\infty)\times (0,u_0)$.
  \item[(c)] For every nonzero subsheaf $A \subseteq F$, either (i) $\mu_f (A) < 0$, or (ii) $\mu_f (A)=0$ and also $\mu_{\Theta + mf}(A) \leq 0$.
  \end{itemize}
\item[(2)] The following are equivalent:
  \begin{itemize}
  \item[(a)] There exists $v_0>0$ such that $F \in \Tc_\omega$ for all $(v,u) \in (v_0,\infty) \times (0,u_0)$.
  \item[(b)]  There exists $v_0>0$ such that, for every nonzero sheaf quotient $F \twoheadrightarrow A$, we have $\mu_\omega (A) > 0$ for all $(v,u) \in (v_0,\infty) \times (0,u_0)$.
  \item[(c)] For any nonzero sheaf quotient $F \twoheadrightarrow A$, either (i) $\mu_f (A)>0$, or (ii) $\mu_f(A)=0$ and $\mu_{\Theta + mf}(A) >0$.
  \end{itemize}
\end{itemize}
\end{lem}

\begin{proof}
The proofs for parts (1) and (2) are essentially the same as those for \cite[Lemma 4.1]{Lo14} and \cite[Lemma 4.3]{Lo14}, respectively, if we replace the slope function $\mu^\ast$ in those proofs by $\mu_{\Theta + mf}$.
\end{proof}

\paragraph[A limit of the heart $\Bc_\omega$] We \label{para:TlFlproperties} now define the following subcategories of $\Coh (X)$:
\begin{itemize}
\item $\Tc^l$, the extension closure of all coherent sheaves satisfying condition (2)(c) in Lemma \ref{lem:TclFcldefinitions}.
\item $\Fc^l$,  the extension closure of all coherent sheaves satisfying condition (1)(c) in Lemma \ref{lem:TclFcldefinitions}.
\end{itemize}
We also define the extension closure in $D^b(X)$
\[
  \Bc^l = \langle \Fc^l [1], \Tc^l \rangle.
\]
Following an argument as in the proof of Lemma \ref{lem:TclFcldefinitions}, it is easy to check that the categories $\Tc^l, \Fc^l$ can equivalently be defined as
\begin{align*}
  \Tc^l &= \{ F \in \Coh (X): F \in \Tc_\omega \text{ for all $v \gg 0$ along \eqref{eq:vhyperbolaequation} } \} \\
  \Fc^l &= \{ F \in \Coh (X): F \in \Fc_\omega \text{ for all $v \gg 0$ along \eqref{eq:vhyperbolaequation} }\}.
\end{align*}
The following immediate properties  are analogous to those  in \cite[Remark 4.4]{Lo14}:
\begin{itemize}
\item[(i)] $\Coh^{\leq 1}(X) \subset \Tc^l$ since  all the torsion sheaves are contained in $\Tc_\omega$, for any ample divisor $\omega$.
\item[(ii)] $\Fc^l \subset \Coh^{=2}(X)$ since every object in $\Fc_\omega$ is a torsion-free sheaf, for any ample divisor $\omega$.
\item[(iii)] $W_{0,\whPhi} \subset \Tc^l$ by the same argument as in \cite[Remark 4.4(iii)]{Lo14}.
\item[(iv)] $f\ch_1(F) \geq 0$ for every $F \in \Bc^l$.  This is clear from the definition of $\Bc^l$ and Lemma \ref{lem:TclFcldefinitions}.  Lemma \ref{lem:TlFlaretorsionpair} below shows that $\Bc^l$ is the heart of a t-structure on $D^b(X)$, and hence an abelian category.  The subcategory
    \[
    \Bl_0 := \{ F \in \Bl : f\ch_1(F)=0 \}
    \]
    is then a Serre subcategory of $\Bl$.
\item[(v)] $\Fc^l \subset W_{1,\whPhi}$.  This follows from (iii) and Lemma \ref{lem:TlFlaretorsionpair} below.
\end{itemize}

\begin{lem}\label{lem:TlFlaretorsionpair}
The pair $(\Tc^l, \Fc^l)$ forms a torsion pair in $\Coh (X)$, and the category $\Bc^l$ is the heart of a bounded t-structure on $D^b(X)$.
\end{lem}

\begin{proof}
By \cite[Lemma 2.5]{Lo7}, we have
\[
  \begin{cases}
  f\ch_1(F) \geq 0 &\text{ if $F \in W_{0,\whPhi}$} \\
  f\ch_1(F) \leq 0 &\text{ if $F \in W_{1,\whPhi}$}
  \end{cases}.
\]
Armed with this observation,  the  argument  in the proof of \cite[Lemma 4.6]{Lo14} applies if we replace  $\mu^\ast$ by $\mu_{\Theta + mf}$ in that proof.
\end{proof}

\begin{lem}\label{lem:ZwtakesBwupperplane}
Fix any $\alpha >0$.  For any nonzero $F \in \Bc^l$, we have $Z_\omega (F) \in \mathbb{H}$ as $v \to \infty$ along the curve \eqref{eq:vhyperbolaequation}.
\end{lem}

\begin{proof}
Part of the proof of $(\Bc_\omega, Z_\omega)$ being a Bridgeland stability condition on $D^b(X)$ \cite[Corollary 2.1]{ABL} asserts that $Z_\omega (F) \in \mathbb{H}$ for any nonzero object $F \in \Bc_\omega$.  This lemma thus follows from the characterisations of $\Tc^l, \Fc^l$ in Lemma \ref{lem:TclFcldefinitions}.
\end{proof}

\paragraph[$Z^l$-stability] \label{para:limit_Bridgeland_stability} We can now define a `limit Bridgeland stability' as follows.  By Lemma \ref{lem:ZwtakesBwupperplane}, for any nonzero object $F \in \Bl$ we know that $Z_\omega (F)$ lies in the upper half plane $\mathbb{H}$ for $v \gg 0$ when $v,u$ lie on the  curve \eqref{eq:vhyperbolaequation}, i.e.\
\[
   (m-\tfrac{e}{2})u^2 + uv =  \alpha+m - e.
\]
We can then define a function germ $\phi (F) : \RR \to (0,1]$ for $v \gg 0$ via the relation
\[
  Z_\omega (F) \in \mathbb{R}_{>0} e^{i \pi \phi (F)(v)} \text{\quad for $v \gg 0$}.
\]
Although $u$ is only an implicit function in $v$ under the constraint \eqref{eq:vhyperbolaequation}, by requiring $u>0$ we can write $u$ as a function in $v$ for $v \gg 0$, in which case $O(u)=O(\tfrac{1}{v})$ as $v \to \infty$.  In fact, we can solve for $u$ as a Laurent series in $\tfrac{1}{v}$ (see \cite[10.4]{Lo20}); alternatively, we can rewrite $Z_\omega (F)$ as  a Laurent polynomial in an indeterminate $v'$ that has the same order of magnitude as $v \to \infty$ (see Remark \ref{rem:changeofvar} below).  When we express $u$ as a Laurent series in $\tfrac{1}{v}$, a further change of variable converts the series to another series whose coefficients have a closed-form formula involving Catalan numbers, allowing us to compute the radius of convergence of the Laurent series \cite{LoWong}.

In any case, we can define a notion of stability as in the case of Bayer's polynomial stability \cite{BayerPBSC}: We say $F$ is $Z^l$-stable (resp.\ $Z^l$-semistable) if, for every $\Bl$-short exact sequence
\[
  0 \to M \to F \to N \to 0
\]
where $M,N \neq 0$, we have
\[
  \phi (M) < \phi (N) \text{ for $v \gg 0$}
\]
(resp.\ $\phi (M) \leq \phi (N)$ for $v \gg 0$).  We will usually write $\phi (M) \prec \phi (N)$ (resp.\ $\phi (M) \preceq \phi (N)$) to mean $\phi (M) < \phi (N)$ for $v \gg 0$ (resp.\ $\phi (M) \leq \phi (N)$ for $v \gg 0$).

\begin{lem}\label{lem:Briminiwallboundedness}
Suppose that there is an object $F \in D^b(X)$ and some  $v_0>0$ such that for all $v >v_0$ along the curve \eqref{eq:vhyperbolaequation} we have that $F\in\Bc_\omega$ and is $Z_\omega$-(semi)stable. Then  $F\in\Bc^l$ and is $Z^l$-(semi)stable.
\end{lem}

\begin{proof}
This follows easily from the definitions of $\Tc^l$ and $\Fc^l$.  (See also \cite[Lemma 7.1]{Lo14})
\end{proof}

\begin{rem}\label{rem:changeofvar}
If we make a change of variables via the `shear matrix'
\[
 \begin{pmatrix} v' \\ u' \end{pmatrix} = \begin{pmatrix} 1 & m-\tfrac{e}{2} \\ 0 & 1 \end{pmatrix} \begin{pmatrix} v \\ u \end{pmatrix}
\]
then the relation \eqref{eq:vhyperbolaequation} can be rewritten as
\[
  m + \alpha = u'v' + e
\]
while $\omega$ can be rewritten as $\omega = u'(\Theta + \tfrac{e}{2}f) + v'f$.  Then $Z_\omega (F)$ is a Laurent polynomial in $v'$, and $Z^l$-stability can equivalently be defined by letting $v' \to \infty$, in which case $Z^l$-stability is indeed a polynomial stability in the sense of Bayer.  Nonetheless, we will use the coordinates $(v,u)$ instead of $(v',u')$ in the rest of this article.
\end{rem}

In the computations that follow, it will be convenient to introduce the following subcategories of $\Tc^l, \Fc^l$:
\begin{align}
  \Tc^{l,+} &= \langle F \in \Coh^{=2}(X) : F \text{ is $\mu_f$-semistable}, \mu_f (F) > 0 \rangle, \notag\\
  \Tc^{l,0} &= \{ F \in \Tc^l : F \text{ is $\mu_f$-semistable}, \mu_f (F)=0 \}, \notag\\
  \Fc^{l,0} &= \{ F \in \Fc^l : F \text{ is $\mu_f$-semistable}, \mu_f (F) =0 \}, \notag\\
  \Fc^{l,-} &= \langle F \in \Coh^{=2}(X) : F \text{ is $\mu_f$-semistable}, \mu_f (F) < 0 \rangle.\label{para:Bltorsiontriplequintuple}
\end{align}
For the same reason as in \cite[Remark 4.8(iii)]{Lo14}, we have the inclusion $\Tc^{l,0} \subset W_{1,\whPhi}$.  Since $W_{0,\whPhi} \subset \Tc^l$ from \ref{para:TlFlproperties}(iii), we have the torsion triple in $\Bl$
\begin{equation}\label{eq:Bltorsiontriple}
  (\Fc^l [1], \,\, W_{0,\whPhi},\,\, W_{1,\whPhi} \cap \Tc^l ),
\end{equation}
which is an analogue of \cite[(4.12)]{Lo14}.  Also, by considering the $\mu_f$-HN filtrations of objects in $\Fc^l$ and $\Tc^l$, we obtain the torsion quintuple in $\Bc^l$
\begin{equation}\label{eq:Bltorsionquintuple}
  (\Fc^{l,0}[1],\,\, \Fc^{l,-}[1],\,\, \Coh^{\leq 1}(X),\,\, \Tc^{l,+},\,\, \Tc^{l,0} ),
\end{equation}
which is an analogue of \cite[(4.13)]{Lo14}.

\paragraph[The category $W_{1,\whPhi} \cap \Tc^l$] From \label{para:W1xTcldecomposition} the torsion  quintuple \eqref{eq:Bltorsionquintuple}, we see that for every object $F \in W_{1,\whPhi} \cap \Tc^l$, the $\Tc^{l,+}$-component must be zero, or else  such a component would contribute a positive intersection number $f\ch_1$; this implies that  $F$ has a two-step filtration $F_0 \subseteq F_1 = F$ in $\Coh (X)$ where $F_0 \in W_{1,\whPhi} \cap \Coh^{\leq 1}(X)$ and is thus a $\whPhi$-WIT$_1$ fiber sheaf, while $F_1/F_0 \in \Tc^{l,0}$.  Since $f\ch_1$ is zero for both $F_0$ and $F_1/F_0$, the transform $\whPhi F[1]$ must be a torsion sheaf.

\paragraph[Transforms of torsion-free sheaves] The \label{para:torsionfreeshtransf}  torsion triple \eqref{eq:Bltorsiontriple} in $\Bl$ is taken by $\whPhi$ to the torsion  triple
\[
  (\wh{\Phi} \Fc^l [1], W_{1,\Phi}, \wh{\Phi} (W_{1,\whPhi} \cap \Tc^l) )
\]
in the abelian category $\wh{\Phi} \Bl$.  This implies that the heart $\wh{\Phi} \Bl [1]$ is a tilt of $\Coh (X)$ with respect to the torsion pair $(\mathcal T, \mathcal F)$ where
\begin{align*}
  \mathcal T &= \wh{\Phi} (W_{1,\whPhi} \cap \Tc^l) [1], \\
  \mathcal F &= \langle \wh{\Phi} \Fc^l [1], W_{1,\Phi} \rangle.
\end{align*}
By \ref{para:W1xTcldecomposition}, we know $\mathcal T \subseteq \Coh^{\leq 1}(X)$.  Consequently, for every torsion-free sheaf $E$ on $X$ we have $E \in \mathcal F \subset \wh{\Phi} \Bl$, which implies ${\Phi} E [1]\in \Bl$.

\paragraph[Phases of objects] We \label{para:phasesofobjects} analyse the phases of various objects in $\Bl$ with respect to $Z^l$-stability.   Note that if $F \in D^b(X)$ satisfies
\begin{align}
    \tilde{n} &= \ch_0(F), \notag\\
  \tilde{d} = f\ch_1(F)&, \,\,\text{\qquad}  \tilde{c} = \Theta\ch_1(F), \notag\\
  \tilde{s} &= \ch_2(F) \label{eq:chFnotation}
\end{align}
then
\begin{align*}
  Z_\omega (F) &= -\ch_2(F) + \tfrac{\omega^2}{2}\ch_0(F) + i \omega \ch_1(F) \\
  &= -\tilde{s} + \left( (m-\tfrac{e}{2})u^2 + uv \right) \tilde{n} + i  ( u(\tilde{c} + m\tilde{d}) + v\tilde{d}) \\
  &= -\tilde{s} + (\alpha + (m-e))\tilde{n} + i ( u(\tilde{c} + m\tilde{d}) + v\tilde{d}) \text{ under the constraint \eqref{eq:vhyperbolaequation}}.
\end{align*}
Now further assume $F$ is a nonzero object of $\Bc^l$.  Consider the following scenarios:
\begin{itemize}
\item[(1)] $F \in \Coh^{\leq 0}(X)$.  Then $\ch_2(F)>0$, and so $Z_\omega (F) \in \mathbb{R}_{<0}$, giving $\phi (F)=1$.
\item[(2)] $F \in \Coh^{\leq 1}(X)$ and $\dimension F = 1$.  Then $\tilde{n}=0$.  We have $\tilde{d} =f\ch_1(F) \geq 0$ in this case.
    \begin{itemize}
    \item[(2.1)] If $\tilde{d}>0$, then $\phi (F) \to \tfrac{1}{2}$.
    \item[(2.2)] If $\tilde{d}=0$, then the effective divisor $\ch_1(F)$  is a positive multiple of the fiber class $f$, and so $(\Theta + mf)\ch_1(F)=\Theta \ch_1(F) = \tilde{c} > 0$, i.e.\ $\Im Z_\omega (F) = u\tilde{c}>0$.
        \begin{itemize}
        \item[(2.2.1)] If $\tilde{s} > 0$ then $\phi (F) \to 1$.
        \item[(2.2.2)] If $\tilde{s} = 0$ then $\phi (F)=\tfrac{1}{2}$.
        \item[(2.2.3)] If $\tilde{s}<0$ then $\phi (F) \to 0$.
        \end{itemize}
    \end{itemize}
\item[(3)] $F \in \Coh^{=2}(X)$ and $f\ch_1(F)=\tilde{d}>0$.  Then $\phi (F) \to \tfrac{1}{2}$.
\item[(4)] $F \in \Tc^{l,0}$.  From the definition of $\Tc^{l,0}$, we have $\tilde{d}=f\ch_1(F)=0$ while $(\Theta + mf)\ch_1(F) >0$; we also know $F$ is $\wh{\Phi}$-WIT$_1$ because $F \in \Fc^l \subset W_{1,\wh{\Phi}}$.  Thus $\wh{F}=\wh{\Phi} F [1]$ is a sheaf of rank zero, and so $\omega \ch_1 (\wh{F})$ must be strictly positive (if $\omega \ch_1(\wh{F})=0$, then $\wh{F}$ would be supported in dimension 0, implying $F$ itself is a fiber sheaf, a contradiction).  Thus from the discussion in \ref{para:ch1realZwcomparison} we know
    \[
    0 < -\Re Z_\omega ({\Phi} \wh{F} [1]) =  \Re Z_\omega (F)
    \]
    and hence $\phi (F) \to 0$.
\item[(5)] $F = A[1]$ where $A \in \Fc^{l,0}$.  Then $f\ch_1(A)=0$ and $(\Theta + mf)\ch_1(A) \leq 0$.  In this case, $A$ is $\wh{\Phi}$-WIT$_1$ by \ref{para:TlFlproperties}(v).  By a similar computation as in (4), we have
    \[
      0 < -\Re Z_\omega ({\Phi} \wh{A} [1])=-\Re Z_\omega (A[1]) = -\Re Z_\omega (F)
    \]
    and so $\phi (F) \to 1$.
\item[(6)] $F = A[1]$ where $A \in \Fc^{l,-}$.  Then $f\ch_1(A)<0$, i.e.\ $f\ch_1(F)>0$.  Hence $\phi (F) \to \tfrac{1}{2}$.
\end{itemize}

\paragraph[Summary] We  \label{para:summarization-diagram}  summarise the constructions in  this section in the following diagram, where a wave type arrow with a pair $(\Tc,\Fc)$ means that (i) such pair is a torsion pair in the source heart and (ii)  the target heart is the tilt at such torsion pair, i.e. the target heart is $\langle \Fc[1], \Tc\rangle$.
\begin{displaymath}
    \xymatrix{ \Coh(X) \ar@<1ex>[rrr]^{\Phi[1]\quad \cong} & & & \Phi(\Coh(X))[1] \ar@<1ex>[lll]^{\wh{\Phi}\quad \cong} & & \\
    & & & & &\Coh(X) \ar@{~>}[llu]|{(W_{0,\whPhi},W_{1,\whPhi})}   \ar@{~>}[lld]|{(\Tc^l,\Fc^l)}
   \ar@{~>}[d]|{(\Tc_\omega,\Fc_\omega)}
    \\
              \wh{\Phi}\Bl \ar@{~>}[uu]|{(\langle  \wh{\Phi} \Fc^l [1], W_{1,\Phi} \rangle, \wh{\Phi} (W_{1,\whPhi} \cap \Tc^l) )}  \ar@<1ex>[rrr]^{\Phi[1]\quad \cong} & & &\Bl  \ar@<1ex>[lll]^{\wh{\Phi}\quad \cong} \ar@{~>}[uu]|{(\langle \Fc^l [1], W_{0,\whPhi}\rangle, W_{1,\whPhi} \cap \Tc^l )} & &\Bc_\omega \ar@{-->}[ll]^{\txt{\tiny limit along curve \\ \tiny (\ref{eq:vhyperbolaequation}) as $v\to\infty$\\
             \tiny  or (\ref{eq:curve-in-lambda-q-plane}) as $\lambda\to 0^+$}}
              }
\end{displaymath}

\section{Slope stability vs limit Bridgeland stability}\label{sec:maintheorem}

Given any torsion-free sheaf $E$ on $X$, we saw in  \ref{para:torsionfreeshtransf} that $\Phi E [1]$ lies in the heart $\Bl$.  In this section, we establish a comparison between   $\mu_\oo$-stability on $E$ and  $Z^l$-stability on the shifted transform $\Phi E [1]$ in the form of Theorem \ref{thm:Lo14Thm5-analogue}, where $\oo$ is taken as \eqref{eq:oonotation}, and $\omega$ is taken as \eqref{eq:omeganotation}.  This theorem is the surface analogue of  \cite[Theorem 5.1]{Lo14}:

\begin{thm}\label{thm:Lo14Thm5-analogue}
Let $p : X \to B$ be a Weierstra{\ss} elliptic surface with base curve $B$.
\begin{itemize}
\item[(A)]  Take B-field $B = \tfrac{e}{2}f$. Suppose $E$ is a $\mu_\oo$-stable torsion-free sheaf on $X$.
\begin{itemize}
\item[(A1)] If $\oo \ch_1^B(E)>0$, then $\Phi E[1]$ is a $Z^l$-stable object in $\Bl$.
\item[(A2)] If $\oo \ch_1^B(E)=0$, then $\Phi E[1]$ is a $Z^l$-semistable object in $\Bl$, and the only $\Bl$-subobjects $G$ of $\Phi E[1]$ where $\phi (G)=\phi (\Phi E [1])$ are objects in  $\Phi (\Coh^{\leq 0}(X))$.
\item[(A3)] If $E$ is locally free, then $\Phi E[1]$ is a $Z^l$-stable object in $\Bl$.
\end{itemize}
\item[(B)] Suppose $F \in \Bl$ is a $Z^l$-semistable object with $f\ch_1(F) \neq 0$, and $F$ fits in the $\Bl$-short exact sequence (which exists by  \eqref{eq:Bltorsiontriple})
\[
 0 \to F' \to F \to F'' \to 0
\]
where $F' \in \langle \Fc^l[1], W_{0,\whPhi} \rangle$ and $F'' \in \langle W_{1,\whPhi} \cap \Tc^l\rangle$.  Then $\whPhi F'$ is a $\mu_\oo$-semistable torsion-free sheaf on $X$.
\end{itemize}
\end{thm}

Note that the objects of $\Phi (\Coh^{\leq 0}(X))$ are precisely direct sums of semistable fiber sheaves of degree 0.

Even though the proof of Theorem \ref{thm:Lo14Thm5-analogue} is analogous to that of \cite[Theorem 5.1(A)]{Lo14}, we include most of the details for ease of reference, and also to lay out explicitly the necessary changes to the proof of \cite[Theorem 5.1]{Lo14}.

\begin{proof}[Proof of Theorem \ref{thm:Lo14Thm5-analogue}(A)]
Let us write $F = \Phi E [1]$  throughout the proof.  Since $\rank (E) \neq 0$, we have $\phi (F) \to \tfrac{1}{2}$.  Take any $\Bl$-short exact sequence
\begin{equation}
0 \to G \to F \to F/G \to 0
\end{equation}
where $G \neq 0$.  This yields a long exact sequence of sheaves
\begin{equation}\label{eq:lem:Lo14Thm5-1A-analogue-eq2}
0 \to \wh{\Phi}^0 G \to E \overset{\alpha}{\to} \wh{\Phi}^0 (F/G) \to \wh{\Phi}^1 G \to 0
\end{equation}
and we see $\whPhi^1 (F/G)=0$.  From the torsion triple \eqref{eq:Bltorsiontriple} in $\Bl$, we know $G$ fits in  the exact triangle
\[
  \Phi (\whPhi^0 G)[1] \to G \to \Phi (\whPhi^1 G) \to \Phi (\whPhi^0 G)[2]
\]
where $\Phi (\whPhi^0 G)[1] \in \langle \Fc^l [1], W_{0,\whPhi}\rangle$ is precisely the $\whPhi$-WIT$_0$ component of $G$, and $\Phi (\whPhi^1 G) \in W_{1,\whPhi} \cap \Tc^l$ the $\whPhi$-WIT$_1$ component of $G$.

\textbf{Suppose $\rank (\image \alpha) = 0$.} Then $\rank (\whPhi^0 G) = \rank E > 0$, and so $f\ch_1 (\Phi (\whPhi^0 G)[1])>0$.  Now we break into two cases:
\begin{itemize}
\item[(a)] $\ch_1 (\image \alpha) \neq 0$.  Then $\mu_{\oo,B} (\whPhi^0 G) < \mu_{\oo,B} (E)$, which implies $\phi (\Phi (\whPhi^0 G)[1]) \prec \phi (F)$.
    \begin{itemize}
    \item[(i)] If $\dimension \Phi (\whPhi^1 G)=2$:  from \ref{para:W1xTcldecomposition} we know $\Phi (\whPhi^1 G)$ fits in a short exact sequence of sheaves
        \begin{equation}\label{eq:PhiwhPhi1G-decomp}
          0 \to A' \to \Phi (\whPhi^1 G) \to A'' \to 0
        \end{equation}
        where $A' \in W_{1,\whPhi} \cap \Coh^{\leq 1}(X) \subset \Coh (\pi)_0$ and $A'' \in \Tc^{l,0}$.  Thus $f\ch_1(\Phi (\whPhi^1 G))=0$, and $Z_\omega (F)$ is dominated by its real part.  From the computation in \ref{para:ch1realZwcomparison}, we know $\Re Z_\omega (\Phi (\whPhi^1 G)) > 0$, and so $\phi (\Phi (\whPhi^1 G)) \to 0$, giving us $\phi (G) \prec \phi (F)$ overall.

    \item[(ii)] If $\dimension \Phi (\whPhi^1 G) \leq 1$: then the component $A''$ in (i) vanishes, and $\Phi (\whPhi^1 G) = A'$ is a $\whPhi$-WIT$_1$ fiber sheaf.  Then
        \[
          Z_\omega (\Phi (\whPhi^1 G))= -\bar{s} + i\bar{c}u
        \]
        where $\bar{s} = \ch_2(A') \leq 0$ while $\bar{c} = \Theta\ch_1(A')\geq 0$.

        If $\bar{s}<0$, then again we  have $\phi (G) \prec \phi (F)$.  On the other hand, if $\bar{s}=0$ then the order of magnitude of $Z_\omega (\Phi (\whPhi^1 G))$ as $v \to \infty$ is  $O(\tfrac{1}{v})$, and so we still have $\phi (G) \prec \phi (F)$ overall.
    \end{itemize}
\item[(b)] $\ch_1 (\image \alpha) =0$.  Then $\image \alpha \in \Coh^{\leq 0}(X)$, in which case $\ch_i (\whPhi^0 G) = \ch_i (E)$ for $i=0,1$.  From the cohomological Fourier-Mukai transform \eqref{eq:cohomFMT-Phi}, it follows that $\ch_0, f\ch_1$ and $\ch_2$ of
 $\Phi (\whPhi^0 G)[1]$ and $F$ agree; from \eqref{eq:ZwPhiEshift} we also see that  all the terms of  $Z_\omega ( \Phi (\whPhi^0 G)[1])$ and $Z_\omega(F)$ agree except the terms involving $u$.  As in (a)(i), we have a decomposition of $\Phi (\whPhi^1 G)$ of the form \eqref{eq:PhiwhPhi1G-decomp}.
    \begin{itemize}
    \item[(i)] If $\dimension \Phi (\whPhi^1 G)=2$: then $A'' \neq 0$, and we have $\Re Z_\omega (A'') > 0$ by \ref{para:ch1realZwcomparison} while $\Im Z_\omega (A'')$ has order of magnitude $O(\tfrac{1}{v})$.  On the other hand, $A'$ is a $\whPhi$-WIT$_1$ fiber sheaf and so $\Re Z_\omega (A') \geq 0$ while $\Im Z_\omega (A')$ also has order of magnitude $O(\tfrac{1}{v})$.  Overall, we have $\phi (G) \prec \phi (F)$.
    \item[(ii)] If $\dimension \Phi (\whPhi^1 G)\leq 1$: then $A''=0$ and $\Phi (\whPhi^1 G)=A'$ is a $\whPhi$-WIT$_1$ fiber sheaf with $\ch_2 (A') \leq 0$. With $\bar{s}, \bar{c}$ as in (a)(ii) above, we observe:
    \begin{itemize}
    \item If $\bar{s}<0$, then $\Re Z_\omega (\Phi (\whPhi^1 G))>0$ while $\Im Z_\omega (\Phi (\whPhi^1 G))$ has magnitude $O(\tfrac{1}{v})$, giving us  $\phi (G) \prec \phi (F)$ overall.
    \item If $\bar{s}=0$, then $\bar{c} \geq 0$ (with $\bar{c}=0$ iff $A'=0$) and $\whPhi^1 G \in \Coh^{\leq 0}(X)$.  Thus $\whPhi^0 (F/G)$ also lies in $\Coh^{\leq 0}(X)$ from the exact sequence \eqref{eq:lem:Lo14Thm5-1A-analogue-eq2}.  Since $F/G \in \Bl$, from the torsion triple \eqref{eq:Bltorsiontriple} in $\Bl$ we know $\whPhi^0 (F/G) \in \langle \whPhi \Fc^l [1], W_{1,\whPhi} \rangle$, i.e.\ $\whPhi^0 (F/G)$ is the extension of a sheaf in $W_{1,\whPhi}$ by a sheaf in $\whPhi \Fc^l [1]$.  However, every nonzero coherent sheaf in $\whPhi \Fc^l [1]$ has $f\ch_1 \neq 0$, and so must be supported in dimension at least 1.  Thus the $\whPhi \Fc^l [1]$-component of $\whPhi^0 (F/G)$ must  vanish, i.e.\ $\whPhi^0 (F/G)$ lies in $W_{1,\whPhi} \cap \Coh^{\leq 0}(X)$, which forces $\whPhi^0 (F/G)$ to be zero.  Then $F/G$ itself is zero, i.e.\ $G=F$.
    \end{itemize}
    \end{itemize}
\end{itemize}

\textbf{Suppose $\rank (\image \alpha) > 0$.} If $\whPhi^0 G \neq 0$ then $0 < \rank (\whPhi^0 G) < \rank (E)$ and so $\mu_{\oo,B} (\whPhi^0 G) < \mu_{\oo,B} (E)$, and so same argument as in part (a) above shows that $\phi (G) \prec \phi (F)$.  From now on, let us assume $\whPhi^0 G =0$, in which case we have the exact sequence of sheaves
\[
  0 \to E \to \whPhi^0 (F/G) \to \whPhi^1 G \to 0.
\]
Thus $G$ is a $\whPhi$-WIT$_1$ object, and from the torsion triple \eqref{eq:Bltorsiontriple} in $\Bl$ we see that $G$ must lie in $W_{1,\whPhi} \cap \Tc^l$.  As in case (a)(i) above, $G$ fits in a short exact sequence in $\Coh (X)$
\[
  0 \to A' \to G \to A'' \to 0
\]
where $A'$ is a $\whPhi$-WIT$_1$ fiber sheaf and $A'' \in \Tc^{l,0}$. We now divide into the following cases:
\begin{itemize}
\item $A'' \neq 0$: then we know $\Re Z_\omega (A'')$ is positive from \ref{para:ch1realZwcomparison} and is $O(1)$, while $\Im Z_\omega (A'')$ is $O(\tfrac{1}{v})$.  On the other hand, since $\ch_2(A') \leq 0$ we know $\Re Z_\omega (A')$ is nonnegative and $O(1)$,  while $\Im Z_\omega (A')$ is  $O(\tfrac{1}{v})$.  Overall, we have $\phi (G) \to 0$, giving us $\phi (G) \prec \phi (F)$.
\item $A''=0$ and $\ch_2 (A')<0$:  then $\phi (G) \to 0$  and we still have $\phi (G) \prec \phi (F)$.
\item $A''=0$ and $\ch_2 (A') =0$:  in this case $A' \in \Phi (\Coh^{\leq 0}(X))$ and so $\phi (G) = \tfrac{1}{2}$.  This is the most intricate of all the cases in this proof to treat, and we single out the following two  scenarios:
    \begin{itemize}
    \item[(S1)] If $\oo \ch_1^B(E)>0$: then $\Re Z_\omega (F)<0$ by \eqref{eq:ooch1twistReZePhiE1relation}, which gives $\phi (F) \succ \tfrac{1}{2} = \phi (G)$.  (Note that this is despite  $\phi (F) \to \tfrac{1}{2}$.)  Therefore, if $\oo \ch_1^B(E)>0$ then $\Phi E[1]$ is always $Z^l$-stable.  This proves statement (A1).
    \item[(S2)] If $\oo \ch_1^B(E)=0$: then $\Re Z_\omega (F)=0$, and $\phi (F) =\tfrac{1}{2} = \phi (G)$.  In this case, $\Phi E[1]$ is $Z^l$-semistable, and it would be strictly $Z^l$-semistable if and only if there exists a $\Bl$-subobject $G$ of $\Phi E[1]$ as in this case. This proves statement (A2).
    \end{itemize}

    Of course,   scenarios (S1) and (S2) above can be ruled out if we impose the vanishing $\Hom (\Phi (\Coh^{\leq 0}(X)),F)=0$, i.e.\ $\Hom (\Phi Q,F)=0$ for every $Q \in \Coh^{\leq 0}(X)$.  Note that for any $Q \in \Coh^{\leq 0}(X)$,
    \[
    \Hom (\Phi Q,F)=\Hom (Q, \whPhi F [1])=\Hom (Q,E[1])=\Ext^1 (Q,E).
    \]
    Hence $\Hom (\Phi (\Coh^{\leq 0}(X)),F)=0$ if and only if $\Ext^1 (Q,E)=0$ for every $Q \in \Coh^{\leq 0}(X)$, which in turn is equivalent to $E$ being a locally free sheaf by Lemma \ref{lem:surfaceshlocallyfreechar} below.  This proves statement (A3), and completes the proof of part (A).
\end{itemize}
\end{proof}

\begin{lem}\label{lem:surfaceshlocallyfreechar}
Suppose $E$ is a torsion-free sheaf $E$ on a smooth projective surface $X$.  Then $E$ is locally free if and only if $\Ext^1 (T,E)=0$ for every $T \in \Coh^{\leq 0}(X)$.
\end{lem}

\begin{proof}
Consider the short exact sequence of sheaves
\[
0 \to E \to E^{\ast \ast} \to Q \to 0
\]
where $Q$ is necessarily a sheaf in $\Coh^{\leq 0}(X)$.  If $E$ is not locally free, then $Q \neq 0$ and we have  $\Ext^1 (Q,E) \neq 0$.  On the other hand, if $E$ is locally free then for any $T \in \Coh^{\leq 0}(X)$ we have $\Ext^1 (T,E)\cong \Ext^1 (E,T\otimes \omega_X) \cong H^1 (X,E^\ast \otimes T)=0$.
\end{proof}

\begin{proof}[Proof of Theorem \ref{thm:Lo14Thm5-analogue}(B)]
Let $F', F, F''$ be as in the statement of the theorem.  We begin by showing that $\whPhi F'$ is a torsion-free sheaf, i.e.\ $\Hom (\Coh^{\leq 1}(X), \whPhi F')=0$, i.e.\
\begin{equation}\label{lem:Lo14Thm5-1B-analogue-eq1}
\Hom (\Phi \Coh^{\leq 1}(X)[1], F')=0.
\end{equation}

Proceeding as in the proof of \cite[Lemma 5.8]{Lo14}, we observe
\begin{align*}
  \Phi \Coh^{\leq 1}(X)[1]  &\subset \langle \{ E \in W_{1,\whPhi} : f\ch_1(E)=0\}, \scalea{\gyoung(;;+,;;+)}[-1], \Coh^{\leq 0}(X)[-1]\rangle [1] \\
  &\subset \langle \Coh (X)[1], \scalea{\gyoung(;;+,;;+)}, \Coh^{\leq 0}(X) \rangle \\
  &\subset \langle \Bl[1], \Bl\rangle.
\end{align*}
Therefore, in order to prove the vanishing \eqref{lem:Lo14Thm5-1B-analogue-eq1}, it suffices to show the following two things:
\begin{itemize}
\item[(i)] For any $G \in W_{1,\whPhi}$ with $f\ch_1(G)=0$, we have $\Hom_{\Bl} (\mathcal{H}^0_{\Bl} (G[1]),F')=0$.
\item[(ii)] $\Hom (\langle\, \scalea{\gyoung(;;+,;;+)}, \Coh^{\leq 0}(X) \rangle, F' )=0$.
\end{itemize}

For (i), let us consider the $(\Tc^l,\Fc^l)$-decomposition of $G$ in $\Coh (X)$
\[
0 \to G' \to G \to G'' \to 0.
\]
This shows $\mathcal{H}^0_{\Bl}(G[1]) = G''[1]$.  Since $G$ is a $\whPhi$-WIT$_1$ sheaf, so is its subsheaf $G'$; thus $G' \in W_{1,\whPhi} \cap \Tc^l$, and  from \ref{para:W1xTcldecomposition} we have  $f\ch_1(G')=0$.  Since $f\ch_1(G)=0$, we also have $f\ch_1(G'')=0$.  By considering the $\mu_f$-HN filtration of $G''$, we obtain  $G'' \in \Fc^{l,0}$.

For any $\Bl$-morphism $\alpha : G''[1] \to F'$ and with $\Ac_1$ defined as in \eqref{eq:A1definition} below, we now have $\image \alpha \in \Ac_1$  and $\phi (\image \alpha) \to 1$ by Lemma \ref{lem:Lo14Lem5-7analogue} below.  However, this gives a composition of $\Bl$-injections
\[
\image \alpha \hookrightarrow F' \hookrightarrow F.
\]
  Hence $\alpha$ must be zero, or else $F$ would be destabilised, proving (i).  A similar argument as above proves (ii).  Hence $\whPhi F'$ is a torsion-free sheaf on $X$.

Next, we show that $\whPhi F'$ is $\mu_\oo$-semistable.  Take any short exact sequence of coherent sheaves on $X$
\[
  0 \to B \to \whPhi F' \to C \to 0
\]
where $B, C$ are both torsion-free sheaves.  Then $\Phi [1]$ takes this short exact sequence to a $\Bl$-short exact sequence
\[
0 \to \Phi B[1] \to F' \to \Phi C [1] \to 0
\]
by \ref{para:torsionfreeshtransf}.  The $Z^l$-semistability of $F$ gives $\phi (\Phi B[1]) \preceq \phi (F)$, which implies $\mu_\oo (B) \leq \mu_\oo (\whPhi F)$.  On the other hand, since $F''$ is precisely the $\whPhi$-WIT$_1$ component of $H^0(F)$, by Lemma \ref{lem:Lo14Lem5-9analogue} below we have $F'' \in \Phi \Coh^{\leq 0}(X)$, i.e.\ $\whPhi F'' \in \Coh^{\leq 0}(X)[-1]$.  This gives
\[
  \mu_\oo (\whPhi F') = \mu_\oo (\whPhi F) \geq \mu_\oo (B).
\]
Hence $\whPhi F'$ is a $\mu_\oo$-semistable torsion-free sheaf.
\end{proof}

Let us define
\begin{equation}\label{eq:A1definition}
  \Ac_1 = \langle \Coh^{\leq 0}(X), \scalea{\gyoung(;;+,;;+)}, \, \Fc^{l,0}[1]\rangle.
\end{equation}

\begin{lem}\label{lem:Lo14Lem5-7analogue}
The category $\Ac_1$ is closed under quotient in $\Bl$, and every object in this category satisfies $\phi \to 1$.
\end{lem}

\begin{proof}
The second part of the lemma follows from the computations in \ref{para:phasesofobjects}.  For the first part, take any $A \in \Ac_1$ and consider any $\Bl$-short exact sequence of the form
\[
0 \to A' \to A \to A'' \to 0.
\]
We need to show that $A'' \in \Ac_1$.  Recall that $\Bl_0 = \{ F \in \Bl : f\ch_1(F)=0\}$ is a Serre subcategory of $\Bl$; also note that $\Ac_1$ is contained in $\Bl_0$.  Hence $A''$ lies in $\Bl_0$, meaning $H^{-1}(A'') \in \Fc^{l,0}[1]$.  On the other hand, since $H^0(A)\in \langle \Coh^{\leq 0}(X), \scalea{\gyoung(;;+,;;+)} \,\rangle$ from the definition of $\Ac_1$, we also have $H^0(A'') \in \langle \Coh^{\leq 0}(X), \scalea{\gyoung(;;+,;;+)}\,\rangle$. Thus $A'' \in \Ac_1$, and we are done.
\end{proof}

\begin{lem}\label{lem:Lo14Lem5-9analogue}
Suppose $F \in \Bl$ is a $Z^l$-semistable object with $f\ch_1(F) \neq 0$.  Then the $\whPhi$-WIT$_1$ component of $H^0(F)$ lies in $\Phi \Coh^{\leq 0}(X)$.
\end{lem}

\begin{proof}
Let $G$ denote the $\whPhi$-WIT$_1$ component of $H^0(F)$.  With respect to the torsion triple \eqref{eq:Bltorsiontriple} in $\Bl$, this is precisely the $W_{1,\whPhi} \cap \Tc^l$ component of $F$.  Hence by \ref{para:W1xTcldecomposition}, $G$ has a two-step filtration $G_0 \subseteq G_1 = G$ in $\Coh (X)$ such that $G_1/G_0 \in \Tc^{l,0}$ and $G_0$ is a $\whPhi$-WIT$_1$ fiber sheaf (and so $\ch_2(G_0) \leq 0$).  Now we have a composition of $\Bl$-surjections
\[
  F \twoheadrightarrow G \twoheadrightarrow G_1/G_0
\]
with $\phi (F) \to \tfrac{1}{2}$ while $\phi (G_1/G_0) \to 0$ from \ref{para:phasesofobjects}(4).  Since $F$ is assumed to be $Z^l$-semistable, this forces $G_1/G_0=0$, and so $G=G_0$.

Suppose now that $\bar{c} = \Theta \ch_1(G)$ and $\bar{s} = \ch_2(G)$.  Then
\[
  Z_\omega (G) = -\bar{s} + i\bar{c}u.
\]
By the $Z^l$-semistability of $F$, the fiber sheaf $G$ cannot have any quotient sheaf with $\ch_2<0$ (such a quotient would have $\phi \to 0$ by \ref{para:phasesofobjects}(2.2.3), destabilising $F$).  Hence $G$ is a slope semistable fiber sheaf with $\ch_2=0$, implying $G \in \Phi \Coh^{\leq 0}(X)$ \cite[Proposition 6.38]{FMNT}.
\end{proof}

\section{The Harder-Narasimhan property of limit Bridgeland stability}\label{sec:HNproperty}

There are two different approaches to proving the Harder-Narasimhan (HN) property of $Z^l$-stability.  The first is a more direct approach, where we decompose the heart $\Ac^l$ using a torsion triple, and then prove that objects in each part of the torsion triple admits a finite filtration.  The second is an indirect approach that relies on a comparison between the large volume limit (as a polynomial stability condition - see \cite[Section 4]{BayerPBSC}) and $Z^l$-stability, and borrowing the HN property of the former stability; this approach is taken in \cite{Lo20}.  In this article, we present the first approach with some of the more routine arguments omitted, namely the proof of Proposition \ref{prop:Alowerstarfinitenessproperties}.  In particular, the first approach follows  the line of thought in \cite[Section 6]{Lo14}.

\begin{lem}\label{lem:Lo14Lem6-1analogue}
The category 
\[
  \Ac_1 = \langle \Coh^{\leq 0}(X), \scalea{\gyoung(;;+,;;+)}, \, \Fc^{l,0}[1]\rangle.
\]
as defined in \eqref{eq:A1definition} is a torsion class in $\Bl$.
\end{lem}

\begin{proof}
We already showed in Lemma \ref{lem:Lo14Lem5-7analogue} that $\Ac_1$ is closed under quotient in $\Bl$.  It remains to show that every object $F \in\Bl$ is the extension of an object in $\Ac_1^\circ$ by an object in $\Ac_1$.

For any $F \in \Bl$, consider the $\Bl$-short exact sequence
\[
0 \to G[1] \to F \to F' \to 0
\]
where $G[1]$ is the $\Fc^{l,0}[1]$-component of $F$ with respect to the torsion quintuple \ref{eq:Bltorsionquintuple}; equivalently, $G$ is the $\Fc^{l,0}$-component of $H^{-1}(F)$.  Note that $\Hom (\Fc^{l,0}[1],F')=0$ by construction.

Suppose $F' \notin \Ac_1^\circ$.  Then there exists a nonzero morphism $\beta : U \to F'$ where $U \in \Ac_1$.  Since $\Ac_1$ is closed under quotient in $\Bl$, we can replace $U$ by $\image \beta$ and assume $\beta$ is a $\Bl$-injection.  The vanishing  $\Hom (\Fc^{l,0}[1],F')=0$ then implies $H^{-1}(U)=0$ and so $U = H^0(U) \in \langle \Coh^{\leq 0}(X), \scalea{\gyoung(;;+,;;+)} \,\rangle$.

Suppose we have an ascending chain in $\Bl$
\[
U_1 \subseteq U_2 \subseteq \cdots \subseteq U_m \subseteq \cdots \subseteq F'
\]
where $U_i \in \langle \Coh^{\leq 0}(X), \scalea{\gyoung(;;+,;;+)}  \,\rangle$ for all $i$.  This induces an ascending chain of coherent sheaves
\[
  \whPhi^0 U_1 \subseteq \whPhi^0 U_2 \subseteq \cdots \subseteq \whPhi^0 F'.
\]
Thus the $U_i$ must stabilise, i.e.\ there exists a maximal $\Bl$-subobject $U$ of $F'$ lying in the extension closure $\langle \Coh^{\leq 0}(X), \scalea{\gyoung(;;+,;;+)}  \,\rangle$.  Applying the octahedral axiom to the $\Bl$-surjections $F \twoheadrightarrow F' \twoheadrightarrow F'/U$ gives the diagram
\[
\scalebox{0.8}{
\xymatrix{
  & & & G[2] \ar[ddd] \\
  & & & \\
  & F' \ar[dr] \ar[uurr] & & \\
  F \ar[ur] \ar[rr] & & F'/U \ar[r] \ar[dr] & M [1] \ar[d] \\
  & & & U [1]
}
}
\]
in which every straight line is an exact triangle, and for some $M \in \Bl$.  The vertical exact triangle gives $H^{-1}(M) \cong G$ and $H^0(M) \cong U$, and so $M \in \Ac_1$.  A similar argument as in the proof of \cite[Lemma 6.1(b)]{Lo14} then shows that $F'/U \in \Ac_1^\circ$, thus  finishing the proof.
\end{proof}

We now define
\begin{align}
\Ac_{1,1/2} &:= \langle \Ac_1, \Fc^{l,-}[1], \scalea{\gyoung(;;+,;;0)},  \scalea{\gyoung(;;*,;+;*)}, \scalea{\gyoung(;+;*,;+;*)} \, \rangle \notag \\
&= \langle \Fc^l[1], \xymatrix @-2.3pc{
\scalea{\gyoung(;;,;;+)} &  \scalea{\gyoung(;;+,;;+)} &  \scalea{\gyoung(;;*,;+;*)} & \scalea{\gyoung(;+;*,;+;*)} \\
& \scalea{\gyoung(;;+,;;0)} & &
}
\rangle.   \label{eq:A1and1over2definition}
\end{align}

\begin{lem}\label{lem:Lo14Lem6-2analogue}
$\Ac_{1,1/2}$ is a torsion class in $\Bl$.
\end{lem}

\begin{proof}
For the purpose of this proof, let us write
\[
  \mathcal{E} =  \xymatrix @-2.3pc{
\scalea{\gyoung(;;,;;+)} &  \scalea{\gyoung(;;+,;;+)} &  \scalea{\gyoung(;;*,;+;*)} & \scalea{\gyoung(;+;*,;+;*)} \\
& \scalea{\gyoung(;;+,;;0)} & &
}.
\]
(Recall that concatenation of 2 by 2 boxes of the form $\scalea{\gyoung(;;,;;)}$ means their extension closure.) It is easy to check that $\mathcal{E}$ is a torsion class in $\Coh (X)$ and that
\[
  \mathcal{E} = \{ H^0(F) : F \in \Ac_{1,1/2}\}.
\]
The same argument as in \cite[Lemma 6.2]{Lo14} then shows that every object in $\Bc^l$ can be written as the extension of an object in $\mathcal{E}$ by an object in $\Ac_{1,1/2}$, proving the lemma.
\end{proof}

Since $\Fc^l[1]$ is contained in $\Ac_{1,1/2}$, any object  $M \in \Bc^l\cap \Ac_{1,1/2}^\circ$ must have $H^{-1}(M)=0$, i.e.\ $M = H^0(M)\in \Tc^l$.  On the other hand, the categories $\Coh^{\leq 1}(X)$ and $\scalea{\gyoung(;+;*,;+;*)}$ are both contained in $\Ac_{1,1/2}$, and so  $W_{0,\whPhi} \subset \Ac_{1,1/2}$.  It follows that
 \begin{equation}\label{eq:A112cCohinW1Tl}
 \Ac_{1,1/2}^\circ \cap \Coh (X) \subset W_{1,\whPhi} \cap \Tc^l.
\end{equation}

Now that we know $\Ac_1, \Ac_{1,1/2}$ are both torsion classes in $\Bl$ with the inclusion $\Ac_1 \subseteq \Ac_{1,1/2}$, we can construct the torsion triple in $\Bl$
\begin{equation}\label{eq:torsiontripleAlowerstar}
  (\Ac_1, \,\Ac_{1,1/2} \cap \Ac_1^\circ, \,\Ac_{1,1/2}^\circ ).
\end{equation}

We have the following finiteness properties for the components of this torsion triple.  We omit the proof of this proposition, since it is modelled after  the proof of the HN property of limit tilt stability on a product elliptic threefold in \cite[Proposition 6.3]{Lo14}:

\begin{prop}\label{prop:Alowerstarfinitenessproperties}
The following finiteness properties hold:
\begin{itemize}
\item[(1)] For $\Ac = \Ac_1$:
\begin{itemize}
\item[(a)] There is no infinite sequence of strict monomorphisms in $\Ac$
\begin{equation}\label{eq:prop-finiteness-eq1}
  \cdots \hookrightarrow E_n \hookrightarrow \cdots \hookrightarrow E_1 \hookrightarrow E_0.
\end{equation}
\item[(b)] There is no infinite sequence of strict epimorphisms in $\Ac$
\begin{equation}\label{eq:prop-finiteness-eq2}
  E_0 \twoheadrightarrow E_1 \twoheadrightarrow \cdots \twoheadrightarrow E_n \twoheadrightarrow \cdots.
\end{equation}
\end{itemize}
\item[(2)] For $\Ac = \Ac_{1,1/2} \cap \Ac_1^\circ$:
    \begin{itemize}
    \item[(a)] There is no infinite sequence of strict monomorphisms \eqref{eq:prop-finiteness-eq1} in $\Ac$.
    \item[(b)] There is no infinite sequence of strict epimorphisms \eqref{eq:prop-finiteness-eq2} in $\Ac$.
    \end{itemize}
\item[(3)] For $\Ac = \Ac_{1,1/2}^\circ$:
  \begin{itemize}
  \item[(a)] There is no infinite sequence of strict monomorphisms \eqref{eq:prop-finiteness-eq1} in $\Ac$.
  \item[(b)] There is no infinite sequence of strict epimorphisms \eqref{eq:prop-finiteness-eq2} in $\Ac$.
  \end{itemize}
\end{itemize}
\end{prop}

Let us now set
\begin{align*}
  \Ac_{1/2} &:= \Ac_{1,1/2} \cap \Ac_1^\circ \\
  \Ac_0 &:= \Ac_{1,1/2}^\circ,
\end{align*}
so that the torsion triple \eqref{eq:torsiontripleAlowerstar} can be rewritten as
\begin{equation}\label{eq:torsiontripleAlowerstar-v2}
(\Ac_1, \, \Ac_{1/2}, \, \Ac_0).
\end{equation}

The following is an analogue of \cite[Lemma 6.5]{Lo14}:

\begin{lem}\label{lem:Lo14Lem6-5analogue}
For $i=1,\tfrac{1}{2}, 0$ and any $F \in \Ac_i$, we have $\phi (F) \to i$.
\end{lem}

\begin{proof}
The case of $i=1$ follows from the definition of $\Ac_1$ and the computation in \ref{para:phasesofobjects}.

For $i=\tfrac{1}{2}$: take any $F \in \Ac_{1/2}$.  If $f\ch_1(F)>0$, then clearly $\phi (F) \to \tfrac{1}{2}$ and we are done.  Let us assume $f\ch_1(F)=0$ from now on.  Then $f\ch_1(H^{-1}(F))=0$, meaning $H^{-1}(F) \in \Fc^{l,0}$; however, $F \in \Ac_1^\circ$ and so $H^{-1}(F)$ must be zero, i.e.\ $F = H^0(F)$.

That $F \in \Ac_{1,1/2} \cap \Coh (X)$ with $f\ch_1(F)=0$ implies $F$ cannot have any subfactors in $\scalea{\gyoung(;;*,;+;*)}$ or $\scalea{\gyoung(;+*,;+*)}$.  Hence $F$ is a fiber sheaf where all the HN factors with respect to slope stability have $\ch_2 \geq 0$.  That  $F \in \Ac_1^\circ$ then forces $F \in \scalea{\gyoung(;;+,;;0)}$, giving us  $\phi (F) = \tfrac{1}{2}$ by  \ref{para:phasesofobjects}(2.2.2).

For $i=0$: take any $F \in \Ac_0$.  From \eqref{eq:A112cCohinW1Tl} we know $F \in W_{1,\whPhi} \cap \Tc^l$.  By \ref{para:W1xTcldecomposition}, we have a two-step filtration $F_0 \subseteq F_1 = F$ in $\Coh (X)$ where $F_0$ is a $\whPhi$-WIT$_1$ fiber sheaf while $F_1/F_0 \in \Tc^{l,0}$.  From \ref{para:phasesofobjects}-(4) we know $\phi (F_1/F_0) \to 0$, so it suffices to show $\phi (F_0) \to 0$.  Since $F \in \Ac_{1,1/2}^\circ$, we have $\Hom (\scalea{\gyoung(;;+,;;0)}, F_0)=0$, implying $F_0 \in \scalea{\gyoung(;;+,;;-)}$.  By \ref{para:phasesofobjects}(2.2.3) we have $\phi (F_0)\to 0$ as desired.
\end{proof}

\begin{lem}\label{lem:Lo14Lem6-6analogue}
An object $F \in \Bl$ is $Z^l$-semistable iff, for some $i=1, \tfrac{1}{2}, 0$, we have:
\begin{itemize}
\item $F \in \Ac_i$;
\item for any strict monomorphism $0 \neq F' \hookrightarrow F$ in $\Ac_i$, we have $\phi (F') \preceq \phi (F)$.
\end{itemize}
\end{lem}

\begin{proof}
Given Lemma \ref{lem:Lo14Lem6-5analogue}, the argument in the proof of \cite[Lemma 6.6]{Lo14}  applies.
\end{proof}

\begin{thm}\label{thm:main1}
The Harder-Narasimhan property holds for $Z^l$-stability on $\Bl$.  That is, every object $F \in \Bl$ admits a filtration in $\Bl$
\[
  F_0 \subseteq F_1 \subseteq \cdots \subseteq F_n = F
\]
where each $F_i/F_{i+1}$ is $Z^l$-semistable, and $\phi (F_i/F_{i-1}) \succ \phi (F_{i+1}/F_i)$ for each $i$.
\end{thm}

\begin{proof}
Using the torsion triple \eqref{eq:torsiontripleAlowerstar-v2},  the finiteness properties in Proposition \ref{prop:Alowerstarfinitenessproperties}, along with Lemma  \ref{lem:Lo14Lem6-6analogue}, the  argument  in the proof of \cite[Theorem 6.7]{Lo14} applies.
\end{proof}

\section{Transforms of 1-dimensional sheaves}\label{sec:trans1dimsh}
In this section,  we study the stability of the Fourier-Mukai transforms of 1-dimensional sheaves. Heuristically, we will need to impose some type of stability on our 1-dimensional sheaves to deduce the $Z^l$-stability of their transforms as in Section \ref{sec:maintheorem}. Luckily, we have another type of stability at our disposal, $Z_{\oo,\frac{e}{2}f}$-semistability, where $\oo$ is as \eqref{eq:oonotation}. Since the Bridgeland slope function for 1-dimensional sheaves becomes
$$
\frac{\ch_2-\frac{e}{2}\ch_1\cdot f}{\beta\ch_1\cdot\widetilde{\omega}},
$$
where $\widetilde{\omega}=\frac{1}{\alpha}(\Theta+mf)+f$, then this type of stability when tested on the subsheaves of a 1-dimensional sheaf does not depend on $\beta$. If $\ch$ is the Chern character of a 1-dimensional sheaf then by \cite[Theorem 1.1]{LQ} we know that the only Bridgeland semistable objects with Chern character $\ch$ for $\beta\gg 0$ are 1-dimensional sheaves and moreover the condition for semistability only needs to be checked on subsheaves. The following definition is in place:

\begin{defn} Consider the $\QQ$-line bundle $L=p^\ast \omega_B /2$. We say a pure 1-dimensional sheaf $\mathcal{E}$ in $\Coh(X)$ is  $L$-twisted $\oo$-Gieseker semistable, or simply twisted Gieseker semistable,  if for every subsheaf $A\hookrightarrow \mathcal{E}$ we have
$$
\frac{\chi_{L}(A)}{\ch_1(A)\cdot \oo}\leq \frac{\chi_{L}(\mathcal{E})}{\ch_1(\mathcal{E})\cdot \oo},
$$
where the $L$-twisted Euler characteristic is defined by
$$
\chi_{L}(E):=\chi(E\otimes L)=\ch_2(E)-\frac{e}{2}\ch_1(E)\cdot f+\ch_0(E)\chi(\mathcal{O}_X)
$$
for every $E\in\Coh(X)$.
\end{defn}
\begin{prop} Let $\mathcal{E}$ be a twisted Gieseker semistable 1-dimensional sheaf with $\chi_L(\mathcal{E})\geq 0$ and $\ch_1(\mathcal{E})\cdot f>0$. Then $\mathcal{E}$ is $\Phi$-WIT$_0$. Moreover, $\Phi(\mathcal{E})$ is torsion-free for $\alpha+m\gg 0$.
\end{prop}
\begin{proof}
Since $W_{0,\Phi}$ is closed under extensions then by using the Jordan-Holder filtration of $\mathcal{E}$ with respect to twisted Gieseker semistability we may assume that $\mathcal{E}$ is twisted Gieseker stable.

Using that $(W_{0,\Phi},W_{1,\Phi})$ is a torsion pair in $\Coh(X)$ we can write a short exact sequence
$$
0\rightarrow E_0\rightarrow \mathcal{E}\rightarrow E_1\rightarrow 0
$$
where $E_i$ is a $\Phi$-WIT$_i$ sheaf for $i=0,1$.

Notice that $E_1\in \Coh^{\leq 1}(X)$ and so $\ch_1(E_1)\cdot f=0$ by \cite[Lemma 6.3]{FMTes}. This implies that $E_1$ is either $0$ or $\ch_1(E_1)=\alpha' f$ for some $\alpha'>0$ since otherwise $E_1$ would be 0-dimensional and therefore $\Phi$-WIT$_0$.

Assume that $E_1\neq 0$. Since $\ch_1(E_1)\cdot f=0$ then $\Phi(E_1)[1]\in \Coh^{\leq 1}(X)$ is a sheaf supported on fibers and therefore
$$
\ch_1(\Phi(E_1)[1])\cdot\Theta=-\ch_2(E_1)\geq 0.
$$
The twisted Gieseker stability of $\mathcal{E}$ implies that
$$
\frac{\chi_L(E_0)}{\ch_1(E_0)\cdot\oo}<\frac{\chi_L(\mathcal{E})}{\ch_1(\mathcal{E})\cdot\oo}=\frac{\chi_{L}(E_0)+\ch_2(E_1)}{\ch_1(\mathcal{E})\cdot\oo}
$$
and therefore
$$
\chi_L(E_0)(\ch_1(E_1)\cdot\oo)<\ch_2(E_1)(\ch_1(E_0)\cdot\oo),
$$
a contradiction since $\chi_L(E_0)\geq 0$. Thus $E_1=0$ and $\mathcal{E}$ is $\Phi$-WIT$_0$.

Now, suppose that $\Phi(\mathcal{E})$ is not torsion-free and let $T$ be its torsion subsheaf so that we have a short exact sequence
$$
0\rightarrow T\rightarrow\Phi(\mathcal{E})\rightarrow F\rightarrow 0
$$
in $\Coh(X)$. Applying ${\whPhi}[1]$ we obtain the distinguished triangle
$$
{\whPhi}(T)[1]\rightarrow \mathcal{E}\rightarrow {\whPhi}(F)[1]\rightarrow {\whPhi}(T)[2].
$$
Since $\mathcal{E}$ is a sheaf then $T$ is ${\whPhi}$-WIT$_1$ and so 1-dimensional. Moreover, $T$ must be supported on fibers, i.e.,
$$
\ch_1(T)=af,\ \ \text{for some}\ \ a>0.
$$
Consider the morphism $g\colon {\whPhi}(T)[1]\rightarrow \mathcal{E}$. The subsheaf $\mathrm{Im}(g)\subseteq \mathcal{E}$ is also 1-dimensional and supported on fibers, i.e.,
$$
\ch_1(\mathrm{Im}(g))=rf\ \ \text{with}\ \ r>0.
$$
Since ${\whPhi}(T)[1]$ is $\Phi$-WIT$_0$ then so is $\mathrm{Im}(g)$. A simple cohomology computation then shows that $\Phi(\mathrm{Im}(g))$ is a subsheaf of $\Phi(\mathcal{E})$ and so must be 1-dimensional and supported on fibers, i.e.,
$$
\ch_1(\Phi(\mathrm{Im}(g)))\cdot\Theta=\ch_2(\mathrm{Im}(g))>0.
$$
Now, from the twisted Gieseker semistability of $\mathcal{E}$ it follows that
$$
\frac{\chi_L(\mathrm{Im}(g))}{\ch_1(\mathrm{Im}(g))\cdot\oo}=\frac{\alpha\ch_2(\mathrm{Im}(g))}{\beta r}\leq \frac{\chi_L(\mathcal{E})}{\ch_1(\mathcal{E})\cdot \oo}.
$$
Fix $m_0>0$ such that $\Theta+m_0 f$ is in the boundary of the nef cone, then $r\leq \ch_1(\mathcal{E})\cdot \Theta+m_0 \ch_1(\mathcal{E})\cdot f$ and so
$$
\frac{\alpha}{\beta \ch_1(\mathcal{E})\cdot(m_0f+\Theta)}\leq \frac{\chi_L(\mathcal{E})}{ \ch_1(\mathcal{E})\cdot\oo },
$$
which is impossible if
$$
\ch_1(\mathcal{E})\cdot\frac{\oo}{\beta}>\frac{\chi_L(\mathcal{E})\ch_1(\mathcal{E})\cdot(\Theta+m_0f)}{\alpha}.
$$
This last inequality is equivalent to
$$
\alpha+m>\frac{\ch_1(\mathcal{E})\cdot\Theta }{\ch_1(\mathcal{E})\cdot f}(\chi_L(\mathcal{E})-1)+m_0\cdot \chi_L(\mathcal{E}).
$$
\end{proof}

\begin{rem} Let $\mathcal{E}$ be a twisted Gieseker semistable 1-dimensional sheaf with $\chi_L(\mathcal{E})\geq 0$ and $\ch_1(\mathcal{E})\cdot f>0$. Notice that if $\alpha+m\gg 0$ then the torsion-free sheaf $\Phi(\mathcal{E})$ is $\mu_f$-semistable. Indeed, if
$$
0\rightarrow E''\rightarrow \Phi(\mathcal{E})\rightarrow E'\rightarrow 0
$$
is a short exact sequence in $\Coh(X)$ then $E''$ is ${\whPhi}$-WIT$_1$ since $\Phi(\mathcal{E})$ is ${\whPhi}$-WIT$_1$. Therefore, by \cite[Lemma 6.2]{FMTes} $\mu_f(E'')\leq 0$.
\end{rem}
\begin{prop}\label{E is in Tau ell}
Let $\mathcal{E}$ be a twisted Gieseker semistable 1-dimensional sheaf with $\chi_L(\mathcal{E})\geq 0$ and $\ch_1(\mathcal{E})\cdot f>0$, and assume that $\alpha+m\gg 0$. Then $\Phi(\mathcal{E})\in \mathcal{T}^{l}$.
\end{prop}
\begin{proof}
Assume for the moment that $\mathcal{E}$ is stable. Since $\Phi(\mathcal{E})$ is $\mu_f$-semistable then by Lemma 3.1 we only need to prove that for every short exact sequence
$$
0\rightarrow E''\rightarrow \Phi(\mathcal{E})\rightarrow E'\rightarrow 0
$$
in $\Coh(X)$ with $\mu_f(E'')=\mu_f(E')=0$ we have $\ch_1(E')\cdot\Theta>0$.

If $E'$ is ${\whPhi}$-WIT$_1$ then ${\whPhi}(E')[1]$ is a quotient of $\mathcal{E}$ and therefore
$$
\frac{\chi_L({\whPhi}(E')[1])}{\ch_1({\whPhi}(E')[1])\cdot \oo}> \frac{\chi_L(\mathcal{E})}{\ch_1(\mathcal{E})\cdot\oo}\geq 0.
$$
This implies that
\begin{align*}
\chi_L({\whPhi}(E')[1])&=-\ch_2({\whPhi}(E'))+\frac{e}{2}\ch_1({\whPhi}(E'))\cdot f\\
&=\ch_1(E')\cdot \Theta+\frac{e}{2}\ch_0(E')+\frac{e}{2}(-\ch_0(E'))\\
&=\ch_1(E')\cdot \Theta>0.
\end{align*}
If $E'$ is not ${\whPhi}$-WIT$_1$ then we know that there is a short exact sequence in $\Coh(X)$
$$
0\rightarrow E_0\rightarrow E'\rightarrow E_1\rightarrow 0
$$
with $E_i$ a ${\whPhi}$-WIT$_i$ sheaf for $i=0,1$. Thus $\ch_1(E_1)\cdot \Theta>0$.

From the $\mu_f$-semistability of $\Phi(\mathcal{E})$ we know that $E'$ is also $\mu_f$-semistable and so by \cite[Lemma 6.2]{FMTes} we conclude that $\ch_1(E_0)\cdot f=\ch_1(E_1)\cdot f=0$. Thus $\ch_0({\whPhi}(E_0))=0$ and so $\ch_1({\whPhi}(E_0))$ is effective. This implies that
$$
\ch_1({\whPhi}(E_0))\cdot f=-\ch_0(E_0)\geq 0.
$$
Therefore $E_0$ is torsion and $\ch_1(E_0)\cdot\Theta\geq 0$ implying that $\ch_1(E')\cdot\Theta>0$.

To conclude the proof, notice that if $\mathcal{E}$ is strictly semistable then $\mathcal{E}$ is in the extension closure of finitely many 1-dimensional stable sheaves each of which is sent via $\Phi$ to an object in $\mathcal{T}^{l}$. Thus $\Phi(\mathcal{E})\in\mathcal{T}^{l}$.
\end{proof}
\begin{thm}\label{assymp_one_dim} Let $\mathcal{E}$ be a twisted Gieseker semistable 1-dimensional sheaf with $\chi_L(\mathcal{E})\geq 0$ and $\ch_1(\mathcal{E})\cdot f>0$, and assume that $\alpha+m\gg 0$. Then $\Phi(\mathcal{E})$ is $Z^{l}$-semistable.
\end{thm}
\begin{proof}
We already know by Proposition \ref{E is in Tau ell} that $\Phi(\mathcal{E})\in\mathcal{T}^{l}$. Suppose that there is a $Z^{l}$-destabilizing sequence in $\mathcal{B}^{l}$ for $\Phi(\mathcal{E})$:
\begin{equation}\label{destabilizing sequence}
0\rightarrow A\rightarrow \Phi(\mathcal{E})\rightarrow B\rightarrow 0.
\end{equation}
We may assume that $B$ is $Z^{l}$-stable. Since $\Phi(\mathcal{E})$ is a sheaf then by analyzing the long exact sequence of cohomology sheaves it follows that $A$ is also a sheaf. We want to show that $B$ is a sheaf as well. Indeed, $B$ fits into an exact sequence in $\mathcal{B}^{l}$
$$
0\rightarrow {H}^{-1}(B)[1]\rightarrow B\rightarrow {H}^0(B)\rightarrow 0.
$$
Since $\Phi(\mathcal{E})\in\mathcal{T}^{l,0}$ then $\phi(\Phi(\mathcal{E}))\rightarrow 0$ along the curve \eqref{eq:vhyperbolaequation} and so $\phi(B)\rightarrow 0$ as well. However, from Section \ref{para:phasesofobjects} we know that
$$
\phi({H}^{-1}(B)[1])>0\ \ \ \text{for}\ \ v\gg 0,
$$
a contradiction to our assumption that $B$ is $Z^{l}$-stable. Thus, ${H}^{-1}(B)[1]=0$ and \eqref{destabilizing sequence} is a short exact sequence of sheaves.

Now, from the triangle
$$
{\whPhi}(A)[1]\rightarrow\mathcal{E}\rightarrow {\whPhi}(B)[1]\rightarrow {\whPhi}(A)[2]
$$
we know that $A$ is ${\whPhi}$-WIT$_1$. Moreover, we obtain the long exact sequence of sheaves
$$
\xymatrix{
0\ar[r] & {\whPhi}^0(B)\ar[r] &{\whPhi}(A)[1]\ar[rr]^{g}\ar@{->>}[dr] & & \mathcal{E}\ar[r] & {\whPhi}^1(B)\ar[r] & 0\\
& & &  M\ar@{^{(}->}[ur] & & &
}
$$
where $M=\mathrm{Im}(g)$. Notice that since $\mathcal{E}$ is 1-dimensional then
$$
\ch_0(M)=0=\ch_0({\whPhi}^1(B)).
$$
From Section \ref{para:phasesofobjects} we know that the $Z^{l}$-destabilizing subobjects of $\Phi(\mathcal{E})$ have $\ch_1(A)\cdot f\leq 0$, but since $A\in\mathcal{T}^{l}$ then $\ch_1(A)\cdot f=0$. Thus
$$
\ch_0({\whPhi}^0(B))=\ch_0({\whPhi}(A)[1])=-\ch_1(A)\cdot f=0.
$$
Since $B$ is a sheaf then the torsion sheaf ${\whPhi}^0(B)$ is $\Phi$-WIT$_1$ and so by \cite[Lemma 6.3]{FMTes}
$$
\ch_1({\whPhi}^0(B))\cdot f=0.
$$
Therefore, $\Phi({\whPhi}^0(B))[1]$ is a torsion sheaf and the short exact sequence of sheaves
$$
0\rightarrow \Phi({\whPhi}^0(B))[1]\rightarrow B\rightarrow \Phi({\whPhi}^1(B))\rightarrow 0
$$
is exact in $\mathcal{T}^{l}$. Moreover, by Section \ref{para:phasesofobjects} we know that unless $\Phi({\whPhi}^0(B))[1]=0$, the phase   $\phi(\Phi({\whPhi}^0(B))[1])\rightarrow \frac{1}{2}$ because $\Phi({\whPhi}^0(B))[1]$ is a fiber sheaf. This is a contradiction since $\phi(B)\rightarrow 0$ and $B$ is $Z^{l}$-stable. Therefore, $\Phi({\whPhi}^0(B))[1]=0$ and
$$
0\rightarrow A \rightarrow \Phi(\mathcal{E})\rightarrow B\rightarrow 0
$$
is a short exact sequence in $W_{1,\whPhi}$, contradicting the twisted Gieseker semistability of $\mathcal{E}$ since
$$
\frac{\ch_2(\whPhi(U)[1])-\frac{e}{2}\ch_1(\whPhi(U)[1])\cdot f}{\ch_1(\whPhi(U)[1])\cdot\oo}=-\frac{\alpha\ch_1(U)\cdot \omega}{\beta(\ch_2(U)-\ch_0(U)\frac{\omega^2}{2})}
$$
for all $U\in D^b(X)$ with $\ch_1(U)\cdot f=0$ along the curve \eqref{eq:vhyperbolaequation}.
\end{proof}

\paragraph[Boundedness of Bridgeland walls via Bogomolov inequalities]
From now on we will assume that the Picard rank of $X$ is 2. This assumption will allow us to bound the walls along the curve \eqref{eq:vhyperbolaequation} and so we will be able to conclude not only $Z^l$-semistability but rather $Z_{\omega}$-stability for the Fourier-Mukai transform of a 1-dimesnional twisted Gieseker semistable sheaf. 

Let us start by recalling the following results about Bogomolov type inequalities on surfaces collected in \cite[Section 6]{MSLectures}:
\begin{lem}\label{constantforbogomolov} Let $X$ be a smooth projective surface and $\omega\in N^1(X)$ be an ample real divisor class. Then there exists a constant $C_{\omega}\geq 0$ such that, for every effective divisor $D\subset X$, we have
$$
C_\omega(D\cdot \omega)^2+D^2\geq 0.
$$
\end{lem}
\begin{defn} Let $X$ be a smooth projective surface and $\omega,B\in N^1(X)$ with $\omega$ ample. For $E\in D^b(X)$ we define
\begin{align*}
\Delta(E)&:= \ch_1(E)^2-2\ch_0(E)\ch_2(E),\\
\bar{\Delta}_{\omega}^{B}(E)&:= (\ch_1^B(E)\cdot\omega)^2-2\ch_0^B(E)\ch_2^B(E)\omega^2,\\
\Delta^C_{\omega,B}(E)&:=\Delta(E)+C_\omega(\ch_1^B(E)\cdot\omega)^2.
\end{align*}
\end{defn}
\begin{thm}\label{bogomolov} Let $X$ be a smooth projective surface and $\omega,B\in N^1(X)$ with $\omega$ ample. Assume that $E$ is $Z_{\omega,B}$-semistable. Then
$$
\bar{\Delta}_{\omega}^{B}(E)\geq 0\ \ \ \text{and}\ \ \ \Delta^C_{\omega,B}(E)\geq 0.
$$
\end{thm}

\begin{lem}\label{lem:nefmoricone}
Let $p : X \to B$ be a Weierstra{\ss} elliptic surface with a section $\Theta$, and suppose $X$ has  Picard rank 2. Then the nef cone $\mathrm{Nef}(X)$ is the set of all non-negative linear combinations of $\Theta+ef$ and $f$, while the cone of effective curves $\overline{\mathrm{NE}}(X)$ (i.e.\ the Mori cone) is the set of all non-negative linear combinations of $f$ and $\Theta$.
\end{lem}
\begin{proof}
The proof for the nef cone is exactly the same as \cite[Proposition V.2.20]{Hartshorne}.  On the other hand, $\overline{\mathrm{NE}}(X)$ is the dual cone of the nef cone $\mathrm{Nef}(X)$. Let $C=Af+B\Theta$ be an effective curve on $X$, then $B=f\cdot C\geq 0$ and $A=(ef+\Theta)\cdot C\geq 0$.
\end{proof}

\begin{prop}\label{constantpicard2} Suppose that $X$ is a Weierstra\ss\ surface of Picard rank 2, and let $\omega=\Theta+mf$ be an ample class. Then every constant
$$
C\geq \frac{e}{(m-e)^2}
$$
satisfies the conditions of Lemma \ref{constantforbogomolov}.
\end{prop}
\begin{proof}
First, note that $D=Af+B\Theta$ is effective if and only if $A\geq 0$ and $B\geq 0$. Clearly, it is enough to bound
$$
\frac{-D^2}{(D\cdot \omega)^2}
$$
when $D^2\leq 0$. Now, $D^2=B(2A-eB)\leq 0$ if and only if $0\leq A\leq \frac{e}{2}B$. Since the same bound will work if we replace $D$ by a multiple of itself then we can assume $B=1$ and allow $A$ to be a rational number. Thus,
$$
\frac{-D^2}{(D\cdot \omega)^2}=\frac{e-2A}{(A+m-e)^2}\leq \frac{e}{(m-e)^2}.
$$
\end{proof}
\begin{rem}\label{constantforboundedness}
Assume that $\omega_0=u_0(\Theta+mf)+v_0f$ is ample and that $C_{\omega_0}$ satisfies the condition of Lemma \ref{constantforbogomolov} for $\omega_0$. Then given $r>0$, the constant  $r^{-2}C_{\omega_0}$ satisfies the condition of Lemma \ref{constantforbogomolov} for $r\omega_0$. Now, since
$$
\frac{e}{u_0^2(m-e)^2}\geq \frac{e}{u_0^2\left(m-e+\frac{v_0}{u_0}\right)^2}
$$
then Proposition \ref{constantpicard2} implies that we can choose
$$
C_{\omega_0}=\frac{e}{u_0^2(m-e)^2}.
$$
\end{rem}

Now, let $E$ be a $Z^{l}$-semistable sheaf in $\mathcal{T}^{l}$ with $\ch_1(E)=\lambda f$ for some $\lambda>0$ and assume $\ch_0(E)>0$, $\ch_2(E)\leq 0$.  Suppose that there is a destabilizing sequence
$$
0\rightarrow A\rightarrow E\rightarrow B\rightarrow 0
$$
in $\mathcal{B}_{\omega_0}$ for some $\omega_0=u_0(\Theta+mf)+v_0f$ along the curve \eqref{eq:vhyperbolaequation} with $0<u_0\ll 1$. Thus, $A\in \mathcal{T}_{\omega_0}$ and so
\begin{equation}\label{categoryequation}
0<\ch_1(A)\cdot \omega_0<\ch_1(E)\cdot \omega_0.
\end{equation}
Along the curve \eqref{eq:vhyperbolaequation} the volume $\omega^2$ equals to a constant $2K$, where $K=\alpha+m-e$. Then the wall equation translates into
\begin{equation}\label{wallequation}
\frac{\ch_1(A)\cdot \omega_0}{\ch_1(E)\cdot \omega_0}=\frac{\ch_2(A)-\ch_0(A)K}{\ch_2(E)-\ch_0(E)K},
\end{equation}
and \eqref{categoryequation} becomes
\begin{equation}\label{categoryequation2}
\ch_2(E)-\ch_0(E)K<\ch_2(A)-\ch_0(A)K<0,
\end{equation}
since $\ch_2(A)-\ch_0(A)K$ and $\ch_2(E)-\ch_0(E)K$ have the same sign and so are negative because of our assumptions on $\ch(E)$.

If $\ch_0(A)=0$ then inequality \eqref{categoryequation2} gives us finitely many values for $\ch_2(A)$. Otherwise, using inequality \eqref{categoryequation} and Theorem \ref{bogomolov} we obtain
\begin{equation}\label{quadraticform2}
\lambda^2u_0^2-4K\ch_0(A)\ch_2(A)> \bar{\Delta}^0_{\omega_0}(A)\geq 0.
\end{equation}
Taking $u_0$ small enough so that $u_0^2<4K$, inequality \eqref{quadraticform2} produces
\begin{equation}\label{boundingch2}
\ch_2(A)<\frac{\lambda^2u_0^2}{4K\ch_0(A)}\leq \lambda^2
\end{equation}
since $A$ is also a sheaf. Combining inequalities \eqref{categoryequation2} and \eqref{boundingch2} we obtain
\begin{equation}\label{boundingch1andch2}
\ch_2(E)-\ch_0(E)K+\ch_0(A)K<\ch_2(A)<\lambda^2
\end{equation}
and therefore $\ch_0(A)$, $\ch_2(A)$, and consequently $\ch_0(B)$ and $\ch_2(B)$ can take only finitely many values.

For convenience of notation, let $\displaystyle S=\frac{\ch_2(A)-\ch_0(A)K}{\ch_2(E)-\ch_0(E)K}$. The wall equation becomes
$$
(\ch_1(A)-S\lambda f)\cdot \omega_0=0,
$$
and therefore the Hodge Index Theorem gives
\begin{equation}\label{boundingch1}
\ch_1(A)^2\leq 2S\lambda\ch_1(A)\cdot f.
\end{equation}
On the other hand, Theorem \ref{bogomolov} and Remark \ref{constantforboundedness} give
\begin{equation}\label{discriminantwithconstant}
-\frac{e}{u_0^2(m-e)^2}S^2\lambda^2u_0^2\leq \Delta(A).
\end{equation}
Combining inequalities \eqref{boundingch1} and \eqref{discriminantwithconstant} we obtain
\begin{equation}\label{finitech1}
-\frac{e}{(m-e)^2}S^2\lambda^2+2\ch_0(A)\ch_2(A)\leq \ch_1(A)^2\leq 2S\lambda\ch_1(A)\cdot f.
\end{equation}
Now, if $\ch_1(A)=\eta f+\gamma \Theta$ then $\ch_1(A)\cdot f=\gamma$ and $\ch_1(A)^2=(2\eta-e\gamma)\gamma$. We will now proceed to analyze inequality \eqref{finitech1} in the following cases:

\textbf{Case 1:} $\gamma <1$. In this case, inequality \eqref{finitech1} produces
$$
-\frac{e}{(m-e)^2}S^2\lambda^2+2\ch_0(A)\ch_2(A)\leq \ch_1(A)^2< 2S\lambda.
$$
Thus, for every pair of values for $\ch_0(A)$ and $\ch_2(A)$ there are finitely many possibilities for $\ch_1(A)^2$. Therefore, since $\ch_1(A)^2=(2\eta-e\gamma)\gamma$ with $\eta$ and $\gamma$ integers then $\ch_1(A)$ can only take finitely many values whenever $\ch_1(A)^2\neq 0$.

When $\ch_1(A)^2=0$ then either $\gamma=0$ and inequality \eqref{categoryequation} implies $0<\eta<\lambda$, or $\eta=\gamma e/2$ and \eqref{categoryequation} implies $0<\gamma K<\lambda u_0^2<16\lambda K^2$. In any case, $\ch_1(A)$ can take only finitely many values.

\textbf{Case 2:} $\gamma\geq 1$. In this case $\ch_1(B)\cdot f<0$. Let $\displaystyle S'=\frac{\ch_2(B)-\ch_0(B)K}{\ch_2(E)-\ch_0(E)K}$, then applying  inequalities \eqref{boundingch1} y \eqref{discriminantwithconstant} to the Bridgeland semistable object $B$ we obtain
\begin{equation}\label{finitech1B}
-\frac{e}{(m-e)^2}S'^2\lambda^2+2\ch_0(B)\ch_2(B)\leq \ch_1(B)^2\leq 0.
\end{equation}
As in Case 1, this implies that $\ch_1(B)$ can take only finitely many values and so does $\ch_1(A)$.

This shows that the Chern character $\ch(A)$ can take only finitely many values and so there are only finitely many walls for the Chern character $\ch(E)=(\ch_0(E),\lambda f,\ch_2(E))$ for $u_0^2<4K$ along the curve \eqref{eq:vhyperbolaequation}, i.e., walls for this Chern character are bounded along the curve \eqref{eq:vhyperbolaequation} for $v\gg 0$.

\begin{cor}\label{cor:main1}
Suppose that $X$ has Picard rank 2. Let $\mathcal{E}$ be a 1-dimensional twisted Gieseker semistable sheaf with $\chi_L(\mathcal{E})\geq 0$ and $\ch_1(\mathcal{E})\cdot f>0$, and assume that $\alpha+m\gg 0$. Then $\Phi(\mathcal{E})$ is $Z_{\omega}$-semistable for $v\gg 0$ along the curve \eqref{eq:vhyperbolaequation}.
\end{cor}

\begin{proof}
By Theorem \ref{assymp_one_dim} we know that $\Phi(\mathcal{E})$ is $Z^l$-semistable. However, since the walls along the curve \eqref{eq:vhyperbolaequation} are bounded for $v\gg 0$ then there exists $v_0$ such that $Z_{\omega}$ semistability coincides with $Z^l$ semistability for all $v>v_0$.  
\end{proof}

%

\section{Asymptotics for Bridgeland Walls on Weierstra{\ss} surfaces}\label{sec:BrivslimitBri}

The boundedness results for Bridgeland mini-walls obtained in Section 6 highly depend on our choice of Chern character $\ch=(\ch_0,\lambda f,\ch_2)$. Indeed, the same techniques will fail if we have $\ch_1=a\Theta$, since $\ch_1\cdot\omega$ will grow as $v\rightarrow\infty$ along the curve \eqref{eq:vhyperbolaequation}. In this section, we want to carefully study the asymptotic behavior of the Bridgeland mini-walls instead of studying all walls at once. Results on boundedness of mini-walls similar to those in Section \ref{sec:trans1dimsh} and \cite{LQ} will then yield, that Bridgeland stability in the outer-most mini-chamber on \eqref{eq:vhyperbolaequation} implies $Z^l$-stability. Combined with Theorem \ref{thm:Lo14Thm5-analogue}(B), this would produce examples of Bridgeland semistable objects whose (inverse) Fourier-Mukai transforms are slope semistable sheaves. In Section \ref{sec:example}, we will give an example where this program is realised.



\paragraph[Polarisation on Weierstra{\ss} surfaces]
Let \label{para:polarisation on Weierstrass surface} $p:X \to B$ be a Weierstra{\ss} surface with a section $\Theta$.  We do not assume that the section $\Theta$ is unique. Recall that $e=-\Theta^2$. Let us fix a positive number $m$ as in Section \ref{sec:limitBridgelandconstr}. In particular $\Theta+mf$ is ample and $f$ is nef. We introduce the following interpolation parameter $\lambda$ for $0<\lambda<1$ and define
\begin{equation}\label{def:H_lambda}
    H_\lambda :=\lambda(\Theta+mf)+(1-\lambda)f.
\end{equation}
So for any polarisation in the affine cone spanned by $\Theta+mf$ and $f$, its direction is uniquely represented by $H_\lambda$. Moreover, we define
\begin{equation}\label{def:H_lambda_perp}
    H^\perp_\lambda:=-\lambda(\Theta+mf)+\left(1+(2m-e-1)\lambda\right)f.
\end{equation}
The choice of coefficients in $H^\perp_\lambda$ makes that
\begin{equation}\label{eqn:ellipticframe}
   H_\lambda.H^\perp_\lambda=0, \qquad g:=H_\lambda.H_\lambda=-H^\perp_\lambda.H^\perp_\lambda. 
\end{equation}
By computation,
\begin{equation} \label{eq:g_in_H_lambda_coordinate}
  g=2\lambda\left(1+(m-\frac{e}{2}-1)\lambda\right) \approx 2\lambda \text{ as } \lambda\rightarrow 0^+.  
\end{equation}

\paragraph[Frame, $(\lambda, w, s, q)$-space in $\Stab (X)$]
We \label{para:frame and 4-dim-space in Stab} refer to Appendix~\ref{sec:Bridgeland wall-chamber structures} for general results and notation on Bridgeland wall-chamber structures, including the definition of a \emph{frame}.  We will use the notation from Appendix~\ref{sec:Bridgeland wall-chamber structures} throughout this section.

For two fixed real numbers $w$ and $\lambda$ with $0<\lambda<1$, we have the frame
\begin{equation} \label{eq-frame-elliptic}
    (H_\lambda, H_\lambda^\perp,w).
\end{equation}

Then for any real numbers $s,q$ satisfying $q>\tfrac{1}{2}s^2$ we can define a Bridgeland stability condition $\sigma_{s,q}$ as in \eqref{eq-sigma-s-q}. Since $\sigma_{s,q}$ depends on the choice of a fixed frame, so it still depends on $\lambda, w$ even though that is suppressed in the notation, see footnote~\ref{footnote: symbol of sigma-s-q}.  As a result, we have the subset of $\Stab (X)$
\begin{equation}\label{eq:Bristab4dimsubsp}
\{ \sigma_{s,q}\in \Stab (X) : (\lambda, w, s, q) \in \mathbb{R}^4,       0<\lambda<1, q>\tfrac{s^2}{2}\}
\end{equation}
which we refer to as the ``$(\lambda, w, s, q)$-space'' in $\Stab (X)$.

\paragraph[Change of variables and the $(\lambda,0,0,q)$-plane in $\Stab (X)$]
Recall \label{para:lambda-q-plane} that we have parameters $u,v \in \mathbb{R}_{>0}$ related by  \eqref{eq:vhyperbolaequation} in the definition of $Z^l$-stability.  We can make the change of variables
\begin{equation}\label{eq-coordinate-change}
  \begin{cases}
  \lambda=\frac{u}{u+v} \\
  t=u+v
  \end{cases}, \text{ or equivalently }
  \begin{cases}
  u=t\lambda \\
  v=t(1-\lambda)
  \end{cases}
\end{equation}
which allows us to write $\omega$ as
\begin{equation} \label{eq:polarisation in two coordinates}
   \omega=u(\Theta+mf)+vf=tH_\lambda. 
\end{equation}

At this point, together with the notation from \ref{para:polarisation on Weierstrass surface}, the parameters $\lambda, t$ correspond to our polarisation $\omega$, while $s,w$ correspond to the $B$-field $B$ (see \eqref{eq-coordinate}). In particular, $\lambda$ parametrises the direction of polarisation $\omega$ in the affine cone spanned by $\Theta+mf$ and $f$.

If we set $B=0$, i.e.\ $s=w=0$, then this forces $q=\tfrac{1}{2}t^2$ in \eqref{eq-q} and  restricts the $(\lambda, w, s, q)$-space in \eqref{eq:Bristab4dimsubsp}  to a ``$(\lambda,0,0,q)$-plane'' (still with the restriction that $0<\lambda<1$ and $q>0$) in $\Stab (X)$.

The volume section \eqref{eq:vhyperbolaequation}, written in terms of $u,v$ as in Figure~\ref{fig:vhyperbolaequation}, can now be written in terms of $\lambda, q$ in the $(\lambda,0,0,q)$-plane as
\begin{equation}\label{eq:curve-in-lambda-q-plane} 
  2q\left(\lambda+(m-\tfrac{e}{2}-1)\lambda^2\right)=\alpha+m-e.
\end{equation}
And we still refer to \eqref{eq:curve-in-lambda-q-plane} as a \emph{volume section}.

\paragraph[Moving frame]
We \label{para:movingframe} still set $B=0$. So a frame $(H_\lambda, H_\lambda^\perp,0)$ is fixed by a choice of real number $\lambda$ with $0<\lambda<1$. 

We will move the frame $(H_\lambda, H_\lambda^\perp,0)$ by varying the parameter $\lambda$. In particular, we let $\lambda \rightarrow 0^+$ along 
\eqref{eq:curve-in-lambda-q-plane}. Then
\[
v\rightarrow +\infty \text{ along \eqref{eq:vhyperbolaequation}} \Longleftrightarrow \lambda\rightarrow 0^+ \text{ along \eqref{eq:curve-in-lambda-q-plane}}
\]
and the volume section \eqref{eq:curve-in-lambda-q-plane} is asymptotic to
\begin{equation}\label{eq:asymptotic-curve-in-lambda-q-plane}
q=\frac{1}{2\lambda}\left(\alpha+m-e\right) \text{ as } \lambda\rightarrow 0^+.
\end{equation}
\begin{figure}
\centering
\begin{tikzpicture}[thick]
        \begin{axis}[
            xmin=0, xmax=1, 
            ymin=0, ymax=40,
            axis x line = center, 
            axis y line = center,
            xtick = \empty,
            ytick = \empty,
            xlabel = {$\lambda<1$},
            ylabel = {$q$},
            xlabel style={below right},
            ylabel style={above left},
            legend style = {nodes=right},
            legend pos = north east,
            clip mode = individual,
            axis line style=black!50,
            ]
            
            \node[below left] at (0,0) { $0$};
            
            \tikzmath{
            \m=3;
            \a=2; 
            \e=2;
            \b=\m+ \a-\e;
            \c=-0.5*\e+\m;
            } 

         \addplot[black!30, samples=280, domain=0.05:0.7] {\b/(2*x)};
         \addplot[black!30, dotted, samples=100, domain=0.01:0.05] {\b/(2*x)};
         \addplot[black, samples=280, domain=0.05:0.7] {\b/((2*\c-2)*x*x+2*x)};
         \addplot[black, dotted, samples=100, domain=0.01:0.05] {\b/((2*\c-2)*x*x+2*x)};
        \end{axis}
\end{tikzpicture}
\caption{The volume section \eqref{eq:curve-in-lambda-q-plane} (black) in the $(\lambda,0,0,q)$-plane (with $0<\lambda<1$ and $q>0$) in $\Stab (X)$ and its asymptotic curve \eqref{eq:asymptotic-curve-in-lambda-q-plane} (gray) as $\lambda\rightarrow 0^+$.}
\label{fig:q-lambda}
\end{figure}
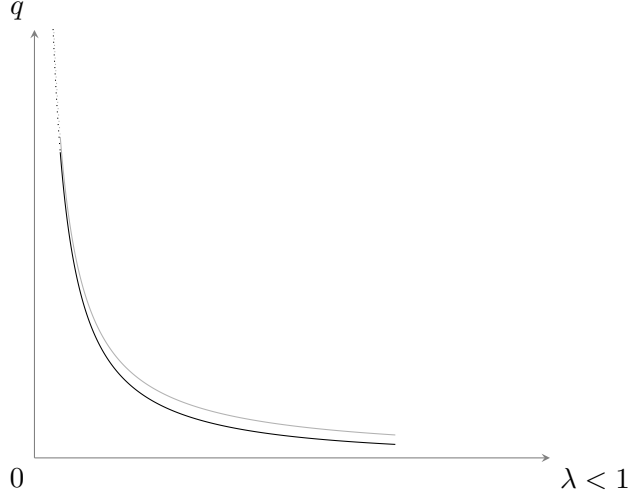

\paragraph[Intersection numbers] Since \label{para:intersectionnumbers}  $\Theta$, $\Theta_i$ are sections, we have the intersection numbers $\Theta_i.f=1$, $\Theta.f=1$. Recall in \ref{para:geometry of X} we have  $\Theta_i^2=\Theta^2=-e$ and $K_X \equiv (2g(B)-2+e)f$.
Let us denote $\theta_i=\Theta.\Theta_i$. Since both $\Theta$ and $\Theta_i$ are irreducible curves, we have $\theta_i\geq 0$. Decomposing $\Theta$, $f$ and $\Theta_i$ with respect to the frame $(H_\lambda, H_\lambda^{\perp}, w)$, we have
$$
\Theta=l_\Theta H_\lambda + l_\Theta^{\perp} H_\lambda^{\perp}, \quad
f=l_f H_\lambda + l_f^{\perp} H_\lambda^{\perp}, \quad
\Theta_i= a_i H_\lambda + b_i H_\lambda^{\perp} + \Delta_i,
$$
where the real coefficients $l_\Theta$, $l_\Theta^{\perp}$, $l_f$, $l_f^{\perp}$, $a_i$, $b_i$ and the class $\Delta_i\in \{ H_\lambda, H_\lambda^{\perp}\}^{\perp}$ are given as follows with $g$ in \eqref{eq:g_in_H_lambda_coordinate}:
\begin{align*}
    l_\Theta g &= \Theta H_\lambda = 1+ (m-e-1)\lambda, &
-l_\Theta^\perp g &= \Theta H_\lambda^\perp = 1+(m-1)\lambda,
\\
l_f g &= f H_\lambda = \lambda, & -l_f^\perp g &= f H_\lambda^\perp = - \lambda,
\\
a_i g &= \Theta_i H_\lambda = 1+(m+\theta_i-1)\lambda, & -b_i g &=  \Theta_i H_\lambda^\perp = 1+(m-\theta_i-e-1)\lambda.
\end{align*}
It is clear that
\begin{align*}
      l_\Theta+l_\Theta^\perp &=-\frac{e\lambda}{g}, & l_f+l_f^\perp &= \frac{2\lambda}{g}, & a_i+b_i &=\frac{2\lambda}{g}(\theta_i+\frac{e}{2}),\\
      l_\Theta-l_\Theta^\perp &=\frac{1}{\lambda}, & l_f-l_f^\perp &= 0, & a_i-b_i &=\frac{1}{\lambda}.
\end{align*}
Basic computation shows that 
\begin{equation} \label{eqn:sumation_indep_lambda}
   a_i H_\lambda + b_i H_\lambda^\perp=\Theta+(\theta_i+e)f.
\end{equation}
Therefore
\begin{equation}\label{eqn:Delta_i}
\Delta_i = \Theta_i -\Theta -(\theta_i+e)f.
\end{equation}
In particular, the divisor class $\Delta_i$ is independent of $\lambda$.

Note that for any numerical invariant $\ch=(\ch_0,\ch_1,\ch_2)$ with $\ch_0\neq 0$, we can write $\ch=e^{L}\left(\ch_0,0,\ch_2-\frac{\ch_1^2}{2\ch_0}\right)$ with $L=\frac{\ch_1}{\ch_0}$. Moreover, $\ch$ is of Bogomolov type as \eqref{eq:Bogomolovtype} if and only if $(\ch_0,0,\ch_2-\frac{\ch_1^2}{2\ch_0})$ is so. For numerical invariants $\ch$ of Bogomolov type, the following proposition gives us the asymptotic behavior of potential walls in the $(\lambda,0,0,q)$-plane as $\lambda\to 0^+$.

\begin{prop}[Potential walls in $(\lambda,0,0,q)$-plane for two-dimensional objects]\label{prop:Potentialwallsinlambda-qplane}
Let $p:X \to B$ be a Weierstra{\ss} surface with a section $\Theta$ and possible other sections $\Theta_i$. Let
$$\ch=\left(\ch_0,\ch_1,\ch_2\right)=(x,0,z)$$
with $\ch_0=x\neq 0$ and $xz\leq 0$ (i.e. $\ch$ is of Bogomolov type as \eqref{eq:Bogomolovtype}).
Write
$$
L=a_L\Theta + b_L f + \sum_i \eta_i \Theta_i
$$
for some $a_L, b_L, \eta_i \in \mathbb{R}$. 
Take the frame $(H_\lambda, H_\lambda^\perp,0)$ as \eqref{eq-frame-elliptic} with $w=0$.
Consider the $(\lambda,0,0,q)$-plane in $\Stab (X)$ in \ref{para:lambda-q-plane}.
Then the potential wall $W(e^{L}\ch, e^{L}\ch')$ with
$$\ch'=\left(\ch_0',\ch_1',\ch_2'\right)=(r, k\Theta+pf+\sum_i \xi_i\Theta_i,\chi),$$
has the following asymptotic behavior in the $(\lambda,0,0,q)$-plane as $\lambda\to 0^+$. Write $e=-\Theta^2$, $\Delta_i$ as \eqref{eqn:Delta_i}, $\Delta_L=\sum_i \eta_i \Delta_i$ and $\Delta'=\sum_i \xi_i \Delta_i$.
\begin{itemize}
\item[(A)] Suppose $k+\sum_i\xi_i=0$ and $p-ek +\sum_i \xi_i\theta_i=0$.
  \begin{itemize}
  \item[(A1)] If $a_L+\sum_i\eta_i=0$ and $b_L-ea_L +\sum_i \eta_i\theta_i=0$, then the potential wall in the $(\lambda,0,0,q)$-plane is the entire region given by $q>0$.
  \item[(A2)] If $a_L+\sum_i\eta_i\neq 0$ or $b_L-ea_L +\sum_i \eta_i\theta_i\neq 0$, then there are no potential walls in the $(\lambda,0,0,q)$-plane.
  \end{itemize}

\item[(B)] Suppose $k+\sum_i\xi_i=0$ and $p-ek +\sum_i \xi_i\theta_i\neq 0$.  Let us set
 \begin{align}
  A &:=-\left(\frac{x\chi-rz}{x}+\Delta'\Delta_L+\Big(a_L+\sum_i\eta_i\Big)\Big(p-\frac{e}{2}k+\sum_i\xi_i(\theta_i+\frac{e}{2})\Big)\right)\frac{a_L+\sum_i\eta_i}{p-ek+\sum_i\xi_i \theta_i}, \label{eq:A-dim-2} \\
  B &:=\frac{z}{x}+\frac{\Delta_L^2}{2}-\left(b_L-ea_L+\frac{e}{2}\Big(a_L+\sum_i\eta_i\Big)+\sum_i\eta_i\theta_i\right)\frac{\frac{x\chi-rz}{x}+\Delta'\Delta_L}{p-ek+\sum_i\xi_i \theta_i}. \label{eq:B-dim-2}
  \end{align}
 \begin{itemize}
  \item[(B1)] If $A\neq 0$
   then the potential wall is asymptotic to
  \begin{equation}\label{eq:wall-asymptotic-type-B1}
  q=\frac{A}{2\lambda^2}, \text{ as } \lambda\to 0^+.
  \end{equation}
  \item[(B2)] If $A = 0$ and $B \neq 0$ then  the potential wall is asymptotic to
  \begin{equation}\label{eq:wall-asymptotic-type-A2}
  q=\frac{B}{2\lambda}, \text{ as } \lambda\to 0^+.
  \end{equation}
  \item[(B3)] If $A=0$ and $B= 0$, then the potential wall is bounded as $\lambda\to 0^+$.
  \end{itemize}
\item[(C)] Suppose $k+\sum_i\xi_i\neq 0$.
  \begin{itemize}
  \item[(C1)] If
  \begin{multline}\label{eq:D-dim-2}
  D:=\\
  \frac{z}{x}+\frac{\Delta_L^2}{2}-\left(\frac{x\chi-rz}{x}+\Delta'\Delta_L+\Big(a_L+\sum_i\eta_i\Big)\Big(p-\frac{e}{2}k+\sum_i\xi_i(\theta_i+\frac{e}{2})\Big)\right)\frac{a_L+\sum_i\eta_i}{k+\sum_i\xi_i}\neq 0,
  \end{multline}
   then the potential wall is asymptotic to
  \begin{equation}\label{eq:wall-asymptotic-type-C1}
  q=\frac{D}{2\lambda}, \text{ as } \lambda\to 0^+.
  \end{equation}
  \item[(C2)] If $D= 0$, then the potential wall is bounded as $\lambda\to 0^+$.
  \end{itemize}
\end{itemize}
\end{prop}
\begin{proof}
Let us break the proof into five steps. 
We are taking $H_\lambda$ as \eqref{def:H_lambda} and taking $H=H_\lambda$ in \eqref{eq-ch-Maciocia}. By the assumption that $\ch_1=0$, we have $y_1=y_2=0$ and $\Delta =0$ in \eqref{eq-ch-Maciocia}.

\emph{Step 1.} Let us decompose the given data with respect to the frame $(H_\lambda, H_\lambda^\perp, w)$ and compute different kinds of intersection numbers. We decompose $L$ according to the frame $(H_\lambda, H_\lambda^\perp, w)$ as
$$L=a_L\Theta + b_L f + \sum_i \eta_i \Theta_i=l_1 H_\lambda+l_2 H^\perp_\lambda+\Delta_L,$$
where $l_1$ and $l_2$ are real coefficients, and the $\RR$-divisor $\Delta_L\in \{H_\lambda, H_\lambda^\perp \}^{\perp}$.
Then
$$
l_1=a_Ll_\Theta+b_L l_f +\sum_i \eta_i a_i,\quad l_2=a_Ll_\Theta^{\perp}+b_L l_f^{\perp}+\sum_i \eta_i b_i, \quad  \Delta_L=\sum_i \eta_i \Delta_i.
$$
In particular, the divisor class $\Delta_L$ is independent of $\lambda$ since $\Delta_i$ is so by \eqref{eqn:Delta_i}.
Recall that $g$ is given in \eqref{eq:g_in_H_lambda_coordinate}. We have
\begin{equation} \label{eq:l_1}
g l_1  = (a_L+ \sum_i\eta_i)+\left(b_L-ea_L +\Big(a_L+ \sum_i\eta_i\Big)(m-1)+\sum_i\eta_i\theta_i\right)\lambda,
\end{equation}
and
\begin{equation}\label{eq:sum_l_1-l_2}
l_1+l_2 = \frac{b_l-\frac{e}{2}a_L+\sum_i \eta_i(\theta_i+\frac{e}{2})}{1+(m-\frac{e}{2}-1)\lambda}, \quad
l_1-l_2 = \frac{a_L+\sum_i \eta_i}{\lambda}.
\end{equation}

Write
$$
\ch_1'=k\Theta+pf+\sum_i\xi_i\Theta_i = c_1 H_\lambda+ c_2 H_\lambda^\perp+\Delta',
$$
with real coefficients $c_1$, $c_2$ and class $\Delta' \in \{H_\lambda, H_\lambda^\perp \}^{\perp}$ as (\ref{eq-ch'-Maciocia}). Then
$$
c_1=k l_\Theta + p l_f +\sum_i \xi_i a_i, \quad
c_2=k l_\Theta^\perp + p l_f^\perp +\sum_i \xi_i b_i, \quad
\Delta'=\sum_i \xi_i \Delta_i.
$$
Hence the divisor class $\Delta'$ is also independent of $\lambda$. We obtain
\begin{equation}\label{eq:gc_1}
g c_1 = (k+\sum_i \xi_i) + \left(p-ek +\Big(k+\sum_i \xi_i\Big)(m-1)+\sum_i \xi_i\theta_i\right)\lambda,
\end{equation}
\begin{equation}\label{eq:sum_c_1-c_2}
g (c_1+c_2) = 2\lambda\Big(p-\frac{e}{2}k+\sum_i\xi_i(\theta_i+\frac{e}{2})\Big).
\end{equation}

\emph{Step 2.} \textbf{Suppose} $k+\sum_i\xi_i=0$ and $p-ek +\sum_i \xi_i\theta_i=0$. Then by \eqref{eq:gc_1}, $c_1=0$, which is independent of $\lambda$.  Now that we have  $y_1=0$ (by assumption) and $c_1=0$, we obtain $xc_1-ry_1=0$. By footnote~\ref{footnote2} in Lemma~\ref{lem:shift by line bundle}, we obtain that the potential wall in the $(\lambda,0,s,q)$-space is given by
$s=l_1$ with $q>\frac{l_1^2}{2}$.

If $a_L+\sum_i\eta_i=0$ and $b_L-ea_L +\sum_i \eta_i\theta_i=0$, then  by \eqref{eq:l_1}, $l_1=0$ and the potential wall in the $(\lambda,0,0,q)$-plane is given by $q>0$. This shows (A1).

If $b_L-ea_L +\sum_i \eta_i\theta_i\neq 0$, then by \eqref{eq:l_1}, $l_1\neq 0$ and  there is no potential wall in the $(\lambda,0,0,q)$-plane. If $a_L+\sum_i\eta_i\neq 0$, then by \eqref{eq:l_1}, $l_1\neq 0$ as $\lambda\to 0^+$ and again there is no potential wall in the $(\lambda,0,0,q)$-plane. This shows (A2).

\emph{Step 3.} We do some computation by assuming that $c_1\neq 0$.  Recall the definition of $P(\ch)$ from Lemma \ref{lem:Bertram}.  Now we have
$$
P(\ch)=(0,\frac{z}{xg}),\quad C(\ch,\ch')=\frac{x\chi-rz}{xg c_1}.
$$
Also, we have $y_1=y_2=0$ and $\Delta=0$ by assumption while $g=\delta$ from \eqref{eqn:ellipticframe}.  Thus by  Lemma~\ref{lem:shift by line bundle},
$$
P(e^{L}\ch)=\left(l_1, \frac{l_1^2-l_2^2}{2}+\frac{z}{xg}+\frac{\Delta_L^2}{2g}\right),$$
and
$$C(e^{L}\ch, e^{L}\ch')=C(\ch,\ch')+l_1-l_2\frac{c_2}{c_1}+\frac{\Delta'\Delta_L}{gc_1}.
$$
The potential wall $W(e^{L}\ch, e^{L}\ch')$ in the $(\lambda,0,s,q)$-space (i.e. $w=0$) is given by
$$
q=\left(\frac{x\chi-rz}{xgc_1}+l_1-l_2\frac{c_2}{c_1}+\frac{\Delta'\Delta_L}{gc_1}\right)(s-l_1)+ \frac{l_1^2-l_2^2}{2}+\frac{z}{xg}+\frac{\Delta_L^2}{2g}.
$$
By restricting to $s=0$, the potential wall $W(e^{L}\ch, e^{L}\ch')$ in the $(\lambda,0,0,q)$-plane is given by
\begin{equation}\label{eq:wall-lambda-q-general}
q=-\left(\frac{x\chi-rz}{x}+\Delta'\Delta_L\right)\frac{gl_1}{gc_1}\frac{1}{g}+l_1l_2\left(\frac{c_1+c_2}{c_1}\right)-\frac{1}{2}(l_1+l_2)^2+\frac{z}{xg}+\frac{\Delta_L^2}{2g}.
\end{equation}

Therefore, by \eqref{eq:l_1}, \eqref{eq:sum_l_1-l_2}, \eqref{eq:gc_1} and \eqref{eq:sum_c_1-c_2}, we have
\begin{multline*}
q=\\
-\left(\frac{x\chi-rz}{x}+\Delta'\Delta_L\right)\frac{(a_L+\sum_i\eta_i)+\Big(b_L-ea_L+(a_L+\sum_i\eta_i)(m-1)+\sum_i\eta_i\theta_i\Big)\lambda}{(k+\sum_i\xi_i)+\Big(p-ek+(k+\sum_i\xi_i)(m-1)+\sum_i\xi_i\theta_i\Big)\lambda} \cdot \frac{1}{g}\\
 +\frac{1}{4}\left(\left(\frac{b_L-\frac{e}{2}a_L+\sum_i \eta_i (\theta_i+\frac{e}{2})}{1+(m-\frac{e}{2}-1)\lambda}\right)^2-\frac{(a_L+\sum_i \eta_i)^2}{\lambda^2}\right)\\
 \quad \cdot \frac{2\lambda\Big(p-\frac{e}{2}k+\sum_i\xi_i(\theta_i+\frac{e}{2})\Big)}{(k+\sum_i\xi_i)+\Big(p-ek+(k+\sum_i\xi_i)(m-1)+\sum_i\xi_i\theta_i\Big)\lambda} \\
  -\frac{1}{2}\left(\frac{b_L-\frac{e}{2}a_L+\sum_i \eta_i (\theta_i+\frac{e}{2})}{1+(m-\frac{e}{2}-1)\lambda}\right)^2+\left(\frac{z}{x}+\frac{\Delta_L^2}{2}\right)\cdot\frac{1}{g}.
\end{multline*}

\emph{Step 4.} \textbf{Suppose} $k+\sum_i\xi_i=0$ and $p-ek +\sum_i \xi_i\theta_i\neq 0$. Then by \eqref{eq:gc_1}, $c_1\neq 0$.
We have
\begin{equation}
q=\frac{A}{2\lambda^2}+\frac{B}{2\lambda}+C(\lambda),
\end{equation}
where $A$ and $B$ are given as (\ref{eq:A-dim-2}) and (\ref{eq:B-dim-2})
and $C(\lambda)$ is bounded as $\lambda\to 0^+$.
The claims in case (B) then follow.

\emph{Step 5.}  \textbf{Suppose} $k+\sum_i\xi_i\neq 0$. Then  by \eqref{eq:gc_1}, $c_1\neq 0$ as $\lambda \to 0^+$. We have
\begin{equation}
q=\frac{D}{2\lambda}+E(\lambda),
\end{equation}
where $D$ is given as in (\ref{eq:D-dim-2})
and $E(\lambda)$ is bounded as $\lambda\to 0^+$.
The claims in case (C) then follow.
\end{proof}

We give a parallel result of Proposition~\ref{prop:Potentialwallsinlambda-qplane} on potential walls in $(\lambda,0,0,q)$-plane for one-dimensional objects in Appendix~\ref{sec:potentialwallsone-dim}.


\section{Transforms of line bundles of fiber degree at least 2}\label{sec:example}

In this section, we combine Theorem \ref{thm:Lo14Thm5-analogue} and the structural results on walls in Section \ref{sec:BrivslimitBri} to prove the following result on sheaves:

\begin{prop}\label{prop:OxaLThetatransformss}
Let $p : X \to B$ be a Weierstra{\ss} elliptic surface such that $X$ has Picard rank two and $e>0$.   Let $m>0$ such that $\Theta + mf$ is ample.  Then for any positive integer $a_L>1$ and real number $\alpha >0$ satisfying
\begin{equation}\label{eq:wallnotlimitcurve}
 \alpha+m -e \neq \tfrac{e}{2}a_L(a_L-1),
\end{equation}
the line bundle $\OO_X(a_L\Theta)$ is $\sigma$-stable for any Bridgeland stability $\sigma$ lying on the curve \eqref{eq:curve-in-lambda-q-plane} on the $(\lambda,0,0,q)$-plane with $\lambda>0$ sufficiently small.  Moreover,  the transform $\whPhi \OO_X(a_L\Theta)$ is a $\mu_\oo$-semistable locally free sheaf of rank $a_L$ where $\oo =  \Theta + (\alpha+m)f$.
\end{prop}

\textbf{Key idea of proof.} The key idea  is that  there is only one wall that is  of the form $W(\ch(\OO_X(a_L\Theta)),-)$, and the condition \eqref{eq:wallnotlimitcurve} ensures that, for $\lambda>0$ sufficiently small, the curve along which we define `limit Bridgeland stability' \eqref{eq:curve-in-lambda-q-plane} either lies above the wall or below the wall.





\begin{lem}\label{lem:OXThetatransf}
For any positive integer $n$, the line bundle $\OO_X(n\Theta)$ is $\Phi$-WIT$_0$, and $\wh{\OO_X(n\Theta)} = \whPhi \OO_X(n\Theta)$ is a locally free sheaf.
\end{lem}

\begin{proof}
For every closed point $s \in B$, the restriction $\OO_X(n\Theta)|_s$ is a line bundle of positive degree on the fiber $X_s$, and hence a $\whPhi_s$-WIT$_0$ sheaf \cite[Proposition 6.38]{FMNT}.  Thus $\OO_X(n\Theta)$ itself is $\whPhi$-WIT$_0$ by \cite[Lemma 3.6]{Lo11}, and the transform $\wh{\OO_X(\Theta)}$ is torsion-free by \cite[Lemma 2.11]{Lo7}.

To see that the transform $\wh{\OO_X(n\Theta)}$ is locally free, take any sheaf $T$ supported in dimension 0; then
\[
  \Ext^1 (T, \wh{\OO_X(n\Theta)}) \cong \Ext^1 (\wh{T}, \OO_X(n\Theta) [-1]) \cong \Hom (\wh{T}, \OO_X(n\Theta)) = 0
\]
where the last equality holds since $\OO_X(n\Theta)$ is torsion-free, and since $T$ is a $\Phi$-WIT$_0$ sheaf whose transform is a fiber sheaf.  Lemma \ref{lem:surfaceshlocallyfreechar} then implies that $\wh{\OO_X(n\Theta)}$ is locally free.
\end{proof}

\paragraph The \label{para:Nonproduct} Weierstra{\ss} elliptic surface $X$ is a product if and only if $\mathbb{L}=\OO_B$ by  \cite[Lemma (III.1.4)]{MirLec}.  Therefore, if $e>0$ then the Weierstra{\ss} surface $X$ cannot be a product.

\begin{lem}\label{lem:degle1pr2uniquesec}
Let $p : X \to B$ be a Weierstra{\ss} elliptic surface with a section $\Theta$, and  suppose $e>0$.
Then $X$ is of Picard rank two if and only if $\Theta$ is the unique section.
\end{lem}

\begin{proof}
Suppose $X$ has Picard rank two.  Then $\mathrm{NS}(X)$ is generated by the class of a section $\Theta'$ and the fiber class $f$ \cite[Theorem (VII.2.1)]{MirLec}.  We will now prove that $\Theta'$ and $\Theta$ are the same curve, and not merely the same curve class.  Suppose
\begin{equation}\label{eq:TTpf}
\Theta'=a\Theta+bf \text{\qquad in  $\mathrm{NS}(X)$}.
\end{equation}
Intersecting with $f$ on both sides of \eqref{eq:TTpf} gives $a=1$.  Squaring both sides of \eqref{eq:TTpf} gives
\[
(\Theta')^2 = \Theta^2 + 2b.
\]
Now, we have $(\Theta')^2=\Theta^2=-e$ by adjunction, and so $b=0$, giving us $\Theta' = \Theta$ in $\mathrm{NS}(X)$.  Thus
\[
  \Theta.\Theta' = \Theta^2 = -e  < 0;
\]
since both $\Theta', \Theta$ are irreducible curves, this implies $\Theta'$ and $\Theta$ are the same curve.  Thus $p$ has a unique section.

Conversely, if $p$ has a unique section $\Theta$, then the Mordell-Weil group $\mathrm{MW}(X)$ of $X$ is trivial.  Then by the Shioda-Tate formula \cite[Corollary (VII.2.4)]{MirLec}, the Picard rank of $X$ must be two.
\end{proof}

\begin{lem}\label{lem:sectionisonlynegcurv}
Let $p : X \to B$ be a Weierstra{\ss} elliptic surface  with a section $\Theta$, and suppose $X$ has  Picard rank two. Suppose also that  $e>0$.  Then    $\Theta$ is the only irreducible negative curve on $X$.
\end{lem}

\begin{proof}
Suppose $C$ is an irreducible negative curve on $X$.  Then $C$ must be extremal in $\overline{\mathrm{NE}}(X)$ by \cite[Lemma 1.22]{KM}.    Lemma \ref{lem:nefmoricone} then implies  either $C \equiv \Theta$ or $C\equiv f$.  Since $e>0$, we have $\Theta^2=-e<0$, i.e.\ $\Theta$ is a negative curve, while $f$ is not.  Hence $C \equiv \Theta$.  Then $C.\Theta = \Theta^2<0$, which in turn implies the curve $C$ coincides with the curve $\Theta$.
\end{proof}
Note that, under the hypotheses of Lemma \ref{lem:sectionisonlynegcurv}, we can also conclude that $\Theta$ must be the unique section, which is the `only if' direction of Lemma \ref{lem:degle1pr2uniquesec}.

\paragraph[An example] An \label{para:assumpuniquesec} example of a Weierstra{\ss} surface $p : X \to B$ such that $X$ has Picard rank two, and where $e >0$, is  an elliptic K3 surface referred to as the Bryan-Leung K3 surface in \cite[Section 2.2]{oberdieck2018}. In this example, we have  $B=\mathbb{P}^1$, $e=2$, and $p$ has exactly 24 singular fibers, all of which are nodal.


\paragraph Suppose \label{para:Arcara-Miles} $p : X \to B$ is a Weierstra{\ss} surface such that $X$ has Picard rank two and $e>0$.  By Lemma \ref{lem:sectionisonlynegcurv}, there is a unique negative curve on $X$, and it is the unique section of $p$ (see also Lemma \ref{lem:degle1pr2uniquesec}).  A theorem of  Arcara-Miles \cite[Theorem 1.1]{ARCARA20161655} now tells us that the only object that could destabilise a line bundle $L$ with respect to a Bridgeland stability in \eqref{eq:Bristab4dimsubsp} is $L(-\Theta)$. Following the notation in Proposition \ref{prop:Potentialwallsinlambda-qplane}, we have
$$
(x,0,z)=(1,0,0)  \text{\quad and \quad} (r,k\Theta+pf,\chi)=(1,-\Theta,-\tfrac{e}{2})
$$
so that $k=-1$. Suppose $L$ is of the form $\OO_X(a_L\Theta)$ with $a_L>1$. By Proposition~\ref{prop:Potentialwallsinlambda-qplane}(C1), the wall $W(\ch(\OO_X(a_L\Theta)), \ch(\OO_X(a_L-1)\Theta))$ is asymptotic to
\begin{equation} \label{eq:Picard rank 2 surface wall of line bundle}
q=\frac{1}{2\lambda} \frac{e}{2}a_L(a_L-1) \text{\quad as $\lambda \to 0^+$}.
\end{equation}


\begin{proof}[Proof of Proposition \ref{prop:OxaLThetatransformss}]
Let $\sigma$ be any Bridgeland stability satisfying the stated hypothesis.  From \ref{para:Arcara-Miles}, We know that $W(\ch(\OO_X(a_L\Theta)), \ch(\OO_X(a_L-1)\Theta))$ is the only wall in the $(\lambda,0,0,q)$-plane for the numerical type of $\OO_X(a_L\Theta)$.  Comparing the asymptotic behaviour of \eqref{eq:curve-in-lambda-q-plane}, namely \eqref{eq:asymptotic-curve-in-lambda-q-plane}, with the asymptotic equation of the wall, namely \eqref{eq:Picard rank 2 surface wall of line bundle}, we see that \eqref{eq:wallnotlimitcurve} ensures $\sigma$ lies in a chamber of Bridgeland stability whenever  $\sigma$ lies on \eqref{eq:curve-in-lambda-q-plane} with $\lambda$ sufficiently small.  (Depending on whether the curve \eqref{eq:curve-in-lambda-q-plane} lies above or below the unique wall as $\lambda \to 0^+$, the Bridgeland stability $\sigma$ lies in either the Gieseker chamber or the other chamber.)

Lemma \ref{lem:Briminiwallboundedness} now implies that $\OO_X(a_L\Theta)$ is $Z^l$-stable.  Since $\OO_X(a_L\Theta)$ is a $\whPhi$-WIT$_0$ sheaf by Lemma \ref{lem:OXThetatransf}, Theorem \ref{thm:Lo14Thm5-analogue}(B) says its transform $\whPhi \OO_X(a_L\Theta)$ is a $\mu_{\oo}$-semistable torsion-free sheaf, which must be locally free by Lemma \ref{lem:OXThetatransf}. For $\oo$ as \eqref{eq:oonotation}, we further take $\beta=\alpha>0$ so that $\oo=\Theta+(\alpha+m)f$.
\end{proof}


%

\paragraph[Comparison with an argument of Bridgeland-Maciocia's]  Contained in the proof of Bridgeland-Maciocia's result \cite[Theorem 1.4]{BMef} is an argument that   shows that the transform $\whPhi \OO_X(a_L\Theta)$ is a $\mu_{\widetilde{\omega}}$-stable torsion-free sheaf for $\widetilde{\omega}$ sufficiently close to the fiber direction, where `sufficiently close' depends on the Chern classes of $\whPhi \OO_X(a_L\Theta)$.  The argument proceeds  as follows: since $\OO_X (a_L\Theta)$ is a torsion-free $\whPhi$-WIT$_0$ sheaf by Lemma \ref{lem:OXThetatransf}, it follows that $\whPhi \OO_X(a_L\Theta)$ is a torsion-free sheaf.  That the restriction of $\whPhi \OO_X(a_L\Theta)$ to the generic fiber of the fibration $p$ is a stable sheaf follows from \cite[Lemma 9.5]{BMef}; then for $\widetilde{\omega} = \Theta + kf$ where $k \gg 0$, we know  $\whPhi \OO_X(a_L\Theta)$ is $\mu_{\widetilde{\omega}}$-stable from the proof of  \cite[Lemma 2.1]{BMef}.

We note that  Bridgeland-Maciocia's approach  begins with a torsion-free sheaf which restricts to a stable sheaf on the generic fiber of the elliptic fibration, while our approach  begins with a limit Bridgeland stable object (which is allowed to be a complex).


\begin{rem}
At first glance,  the statement of Proposition \ref{prop:OxaLThetatransformss} appears to be similar to that of \cite[Theorem 4.4]{LZ3}, which says that on a Weierstra{\ss} threefold $p : X \to S$ where $X$ is $K$-trivial and $K_S$ is numerically $K$-trivial, any line bundle of nonzero fiber degree on $X$ is taken by a  Fourier-Mukai transform to a $\mu_{\widetilde{\omega}}$-stable locally free sheaf, for any polarisation $\widetilde{\omega}$.  One quickly finds, however, that the argument in \cite{LZ3} does not carry over directly to the situation of Proposition \ref{prop:OxaLThetatransformss}.  A technical reason is that the base of the fibration in Proposition \ref{prop:OxaLThetatransformss} is $\mathbb{P}^1$, which is not numerically $K$-trivial.
\end{rem}

\paragraph In proving Proposition \ref{prop:OxaLThetatransformss}, we relied on Arcara-Miles' result that there is only one possible destabilising object for a line bundle, if the surface contains a unique negative curve.  This is only one half of their theorem \cite[Theorem 1.1]{ARCARA20161655}; the other half of their theorem states that the result holds also for surfaces with no negative curves (such as  $C \times \mathbb{P}^1$ where $C$ is an elliptic curve).  For such and other surfaces for which Arcara-Miles' theorem holds, it seems plausible that an analogue of Proposition \ref{prop:OxaLThetatransformss} would hold.

%

\appendix

\section{Bridgeland wall-chamber structures}
Let \label{sec:Bridgeland wall-chamber structures} $X$ be a smooth projective surface. We briefly recall the wall-chamber structures in the Bridgeland stability manifold $\Stab(X)$. We will consider the  stability conditions $\sigma_{\omega,B}$ defined in \ref{para:Bridgestabonsurfaces}. Our study of  wall and chamber structures consists of two steps: (i) We fix a `frame' and write $\omega$ and $B$ with respect to the frame as in \eqref{eq-coordinate}, and study potential walls; (ii) we move the frame. Step (i) follows the work of Maciocia \cite{Mac2}. We give an example of step (ii) on elliptic surfaces in \ref{para:movingframe}, by varying a parameter $\lambda$.

By \emph{fixing a frame}, we  mean that we fix a triple $(H, H^\perp,w)$ where $H$ is an ample $\mathbb{R}$-divisor on $X$,   $H^\perp$ is an $\mathbb{R}$-divisor satisfying $H.{H^\perp}=0$, and $w$ is  a real number.  The divisor $H^\perp$ is taken to be zero  if the Picard number of $X$ is one.  In general, the divisor  $H^\perp$ is not unique even up to a scalar multiple if the  Picard number of $X$ is bigger than two. We set
\begin{equation}\label{eq-g-d}
g:=H.H, \quad \delta:=-H^\perp.H^\perp.
\end{equation}
The Hodge Index Theorem implies that $\delta\geq 0$, and $\delta=0$ if and only if ${H^\perp}=0$.

Having  fixed a frame $(H,{H^\perp},w)$, we can then set
\begin{equation}\label{eq-coordinate}
\begin{cases}
\omega:=tH\\
B:=sH+w H^\perp
\end{cases} \text{\quad where $t \in \mathbb{R}_{>0}, s \in \mathbb{R}$}
\end{equation}
and think of $\omega, B$ as depending on $t, s$, respectively.  By varying $w$, we then obtain a $w$-indexed family of $(s,t)$ half-planes in $\Stab (X)$:
$$\Pi_{(H,{H^\perp},w)}:=\{ \sigma_{tH,sH+wH^\perp}\ |\ t \in \mathbb{R}_{>0}, s \in \mathbb{R} \} \subset \Stab(X).$$

Let $\ch=(\ch_0,\ch_1,\ch_2)$ be a fixed Chern character. Following the notations of Maciocia \cite{Mac2}, we rewrite it with respect to the frame $(H,{H^\perp},w)$ as \footnote{ Here the notations are different from \ref{para:cohomFMTformulas}, see footnote~\ref{footnote-different-notations} in \ref{para:cohomFMTformulas}.}
\begin{equation}\label{eq-ch-Maciocia}
\ch=(\ch_0,\ch_1,\ch_2)=(x,y_1 H+ y_2 {H^\perp} + \Delta,z)
\end{equation}
for some real coefficients  $y_1, y_2$ and $\Delta \in \{ H,{H^\perp} \}^{\perp}$, i.e. $\Delta$ is an $\mathbb{R}$-divisor satisfying
$\Delta.H=0$ and $\Delta.{H^\perp}=0$. 
Similarly, we write the potentially destabilising Chern character with respect to the frame as
\begin{equation}\label{eq-ch'-Maciocia}
\ch'=(\ch_0',\ch_1',\ch_2')=(r,c_1 H+ c_2 {H^\perp}+\Delta',\chi)
\end{equation}
for some real coefficients  $c_1, c_2$ and $\Delta' \in \{ H,{H^\perp} \}^{\perp}$. For fixed $\ch, \ch'$, the corresponding  \emph{potential wall} is defined as stability conditions where objects in the heart of characters $\ch$ and $\ch'$ have the same phase, i.e.
$$W(\ch,\ch'):=\{\sigma = (\Bc, Z) \in \Stab(X)|\, \Re Z(\ch) \Im Z(\ch')-\Re Z(\ch') \Im Z(\ch)=0\}.$$
In the notation $\sigma = (\Bc,Z)$ above for a Bridgeland stability, $\Bc$ is a heart and $Z$ is the central charge of the stability condition.  A potential wall $W(\ch,\ch')$ is a \emph{Bridgeland wall} if there is a $\sigma=(\Bc, Z)\in W(\ch,\ch')$ together with $\sigma$-semistable objects $G\subset F\in\Bc$ such that $\ch(F)=\ch$, $\ch(G)=\ch'$.

Fix a frame $(H,{H^\perp},w)$. Following the idea of Li-Zhao \cite{LZ19}, we define $\sigma'_{\omega,B}=(Z'_{\omega,B}, \Bc'_{\omega, B})$ as the right action of $\begin{pmatrix}
1 & 0\\
-\frac{s}{t} & \frac{1}{t}
\end{pmatrix}$ on $\sigma_{\omega,B}$, i.e. $\Bc'_{\omega, B}:=\Bc_{\omega, B}$ and
\[
  (\Re Z'_{\omega,B}, \Im Z'_{\omega,B}) := (\Re Z_{\omega,B}, \Im Z_{\omega,B} ) \begin{pmatrix} 1  & 0 \\ -\tfrac{s}{t} & \tfrac{1}{t} \end{pmatrix}.
\]
Thus
\begin{equation} \label{eq-Z-Z'}
Z'_{\omega,B}(F)=\left(\Re Z_{\omega,B}(F)-\tfrac{s}{t}\Im Z_{\omega,B}(F)\right)+\tfrac{1}{t}i\Im Z_{\omega,B}(F).
\end{equation}
By varying $w$ again, we obtain another $w$-indexed family of half planes with coordinates $(s,t)$ (which is different from the $\Pi_{(H,H^\perp,w)}$ defined earlier):
$$\Pi'_{(H,{H^\perp},w)}:=\{ \sigma'_{tH,sH+wH^\perp}\ |\ t \in \mathbb{R}_{>0}, s \in \mathbb{R} \} \subset \Stab(X).$$
\begin{lem}
Fix a frame $(H,{H^\perp},w)$. The above right action identifies the potential walls $W(\ch,\ch')$ in the $(s,t)$-plane $\Pi_{(H,{H^\perp},w)}$ with the potential walls $W(\ch,\ch')$ in the $(s,t)$-plane $\Pi'_{(H,{H^\perp},w)}$.
\end{lem}
\begin{proof}
\cite[Lemma 2.6]{Liu18}.
\end{proof}

Fix a frame $(H,{H^\perp},w)$.   We introduce $(s,q)$-coordinates in addition to $(s,t)$-coordinates via the change of variables
\begin{equation} \label{eq-q}
q:=\frac{s^2+t^2}{2}
\end{equation}
(note that $t>0$).  This way, there is a bijection between the `$(s,t)$-plane'
\[
  \{ (s,t) : s \in \mathbb{R}, t \in \mathbb{R}_{>0} \}
\]
and the `$(s,q)$-plane'
\[
  \{ (s,q) : s \in \mathbb{R}, q \in \mathbb{R}_{>0}, q > \tfrac{1}{2}s^2 \}.
\]
The family $\Pi'_{(H,{H^\perp},w)}$ of $(s,t)$-planes will be referred to  as  the family  $\Sigma_{(H,{H^\perp},w)}$ when using  $(s,q)$-coordinates.

The advantage of the $(s,q)$-coordinate is that potential walls will be semi-lines (instead of semi-circles in the $(s,t)$-coordinate). We will write \footnote{ \label{footnote: symbol of sigma-s-q} Note that $\sigma_{s,q}$ still depends on the choice of a frame $(H,{H^\perp},w)$ even though that is suppressed in the notation.}
\begin{equation} \label{eq-sigma-s-q}
\sigma_{s,q}:=\sigma'_{tH,sH+wH^\perp}.
\end{equation}
The associated central charge, given by  (\ref{eq-Z-Z'}), can be rewritten in $(s,q)$-coordinates as
\begin{eqnarray}\label{eq-central-charge}
Z_{s,q}(F)&:=&(-\ch_2(F)+\ch_0(F) g q)+\left(\tfrac{1}{2}\ch_0(F) \delta w^2 + w \ch_1(F).{H^\perp}\right)\nonumber\\
& &+i(\ch_1(F).H-\ch_0(F) g s).
\end{eqnarray}

\begin{table}
\caption {A summary of notations for $(s,t)$- and $(s,q)$-planes after fixing a frame $(H,{H^\perp},w)$. Here we take $(\omega, B)$ as in \eqref{eq-coordinate}, $s, t \in \mathbb{R}$ with $t>0$ and $q$ is given by \eqref{eq-q}. In particular, $q>\frac{1}{2}s^2$.}
\begin{center}
\begin{tabular}{ c | c | c }
\hline
$(s, t)$-plane $\Pi_{(H,{H^\perp},w)}$ & $(s, t)$-plane $\Pi'_{(H,{H^\perp},w)}$ & $(s, q)$-plane $\Sigma_{(H,{H^\perp},w)}$ \\
$\sigma_{tH,sH+wH^\perp}$ & $\sigma'_{tH,sH+wH^\perp}$ & $\sigma_{s,q}:=\sigma'_{tH,sH+wH^\perp}$ \\
\eqref{eq:Zwformula} & \eqref{eq-Z-Z'} & \eqref{eq-central-charge} \\
\hline
\end{tabular}
\end{center}
\end{table}

We call a Chern character $\ch=\left(\ch_0,\ch_1,\ch_2\right)$ of Bogomolov type if
\begin{equation}
    \label{eq:Bogomolovtype}
    \ch_1^2-2\ch_0\ch_2\geq 0.
\end{equation}

\begin{lem}[Bertram's nested wall theorem in $(s,q)$-plane]\label{lem:Bertram} Fix a Chern character $\ch$ of Bogomolov type. Fix a frame $(H,{H^\perp},w)$ and denote $g$, $\delta$ as (\ref{eq-g-d}). Use the notations for $\ch$, $\ch'$ as (\ref{eq-ch-Maciocia}) and
(\ref{eq-ch'-Maciocia}).
\begin{itemize}
\item[(A)]  Suppose $x\neq 0$. Then all  potential walls $W(\ch,\ch')$   in the $(s,q)$-plane $\Sigma_{(H,{H^\perp},w)}$ are given by semi-lines passing through the same point $P(\ch):=\left(\frac{y_1}{x}, \frac{1}{2}\left(\frac{y_1^2}{x^2}-F(\ch)\right)\right)$ with slopes $C(\ch,\ch')$ \footnote{ \label{footnote1} We use the convention that if $xc_1-ry_1=0$, then the slope is infinite and the wall is the semi-line $s=\frac{y_1}{x}$ with $q>\frac{s^2}{2}$.}:
\begin{equation}\label{eq-wall-passing-fixed-point}
q=C(\ch,\ch')\left(s-\frac{y_1}{x}\right)+\frac{1}{2}\left(\frac{y_1^2}{x^2}-F(\ch)\right), \qquad (q>\tfrac{s^2}{2}),
\end{equation}
where
\begin{eqnarray}
C(\ch,\ch')&:=&\frac{x\chi-rz +w\delta (xc_2-ry_2)}{g(xc_1-ry_1)}, \label{eq-Csigma}\\
F(\ch)&:=&\frac{\delta}{g}\left(w-\frac{y_2}{x}\right)^2+\frac{1}{x^2 g}(y_1^2 g - y_2^2 \delta  - 2xz)\geq 0. \label{eq-F}
\end{eqnarray}
In particular,  $P(\ch)$ is on or below the parabola $q=\frac{s^2}{2}$.

\item[(B)] Suppose $x=0$ and $\ch_1 H>0$ (i.e. $y_1>0$). If $r=0$, then the potential wall is given by $y_1\chi=z c_1$, and there is no potential wall in the $(s,q)$-plane. If $r\neq 0$, then all potential walls $W(\ch,\ch')$  in the $(s,q)$-plane are given by semi-lines of the same slope $C=C(\ch)$, and they  pass through points of the form $P'(\ch'):=(\frac{c_1}{r}, \frac{1}{2}\left(\frac{c_1^2}{r^2}-F'(\ch')\right))$:
\begin{equation} \label{eq:wall_dim-1}
q=C(\ch)(s-\frac{c_1}{r})+\frac{1}{2}\left(\frac{c_1^2}{r^2}-F'(\ch')\right), \qquad (q>\tfrac{s^2}{2}),
\end{equation}
where
\begin{eqnarray}
C(\ch)&:=&\frac{z+\delta w y_2}{g y_1},\label{eq-C-dim-1}\\
F'(\ch')&:=&\frac{\delta}{g}(w-\frac{c_2}{r})^2+\frac{1}{r^2 g}(c_1^2 g- c_2^2 \delta-2r\chi). \label{eq-F'}
\end{eqnarray}
 Moreover, if $\ch'$ is also of Bogomolov type, then $F'(\ch')\geq 0$ and $P'(\ch')$ is on or below the parabola $q=\frac{s^2}{2}$.
\end{itemize}
\end{lem}
\begin{proof}
\cite[Lemma 2.8]{Liu18}.
\end{proof}

\begin{lem}[Shift by line bundle] \label{lem:shift by line bundle}
Fix a Chern character $\ch$ of Bogomolov type. Fix a frame $(H,{H^\perp},w)$ and use the notations above. Fix an $\mathbb{R}$-divisor $L$ of the form
\[ L=l_1 H+l_2 H^\perp+\Delta_L
\]
with real coefficients $l_1$ and $l_2$, and $\Delta_L\in \{H,H^\perp \}^\perp$ in $\mathrm{NS}_{\mathbb{R}}(X)$.
\begin{itemize}
\item[(A)] Suppose $x\neq 0$. Then potential walls of the form $W( e^{L}\ch, e^{L}\ch')$ in the $(s,q)$-plane are all given by semi-lines passing through the same point
\[
 P(e^{L}\ch) = P(\ch) + \left( l_1, \frac{1}{2}l_1^2+\frac{y_1}{x}l_1-\frac{\delta}{2g}l_2^2+\frac{\delta}{g}(w-\frac{y_2}{x})l_2+\frac{1}{2g}\Delta_L^2+\frac{\Delta\Delta_L}{xg}\right)
\]
with slopes \footnote{\label{footnote2} We use the convention that if $xc_1-ry_1=0$, then the slope is infinite and the wall is the semi-line $s=\frac{y_1}{x}+l_1$ with $q>\frac{s^2}{2}$.}
\begin{equation}\label{eq:slope-shift}
 C(e^{L}\ch, e^{L}\ch')= C(\ch,\ch')+l_1-l_2\frac{\delta}{g}\frac{xc_2-ry_2}{xc_1-ry_1}+\frac{x\Delta'\Delta_L-r\Delta\Delta_L}{g(xc_1-ry_1)}
\end{equation}
in the region $q>\frac{s^2}{2}$.

\item[(B)] Suppose $x=0$ and $\ch_1 H>0$. Then potential walls of the form $W( e^{L}\ch, e^{L}\ch')$ in the $(s,q)$-plane are all given by semi-lines passing through points
\[
 P'(e^{L}\ch') = P'(\ch') + \left( l_1, \frac{1}{2}l_1^2+\frac{c_1}{r}l_1-\frac{\delta}{2g}l_2^2+\frac{\delta}{g}(w-\frac{c_2}{r})l_2+\frac{1}{2g}\Delta_L^2+\frac{\Delta'\Delta_L}{rg}\right)
\]
with the same slope
\begin{equation}\label{eq:ssdo}
 C(e^{L}\ch)= C(\ch)+l_1-l_2\frac{\delta}{g}\frac{y_2}{y_1}+\frac{\Delta \Delta_L}{gy_1}
\end{equation}
in the region $q>\frac{s^2}{2}$.
\end{itemize}
\end{lem}
\begin{proof}
The formula \eqref{eq:slope-shift} follows from \eqref{eq-Csigma}.
By using formula \eqref{eq-F}, we get
$$F(e^L\ch)= F(\ch)+\frac{\delta}{g}l_2^2-\frac{2\delta}{g}(w-\frac{y_2}{x})l_2-\frac{\Delta_L^2}{g}-\frac{2\Delta\Delta_L}{xg}.
$$
Thus we obtain the formula for $P(e^{L}\ch)$. This shows part (A).
The formula \eqref{eq:ssdo} follows from \eqref{eq-C-dim-1}. By using formula \eqref{eq-F'}, we get
$$F'(e^L\ch')= F'(\ch')+\frac{\delta}{g}l_2^2-\frac{2\delta}{g}(w-\frac{c_2}{r})l_2-\frac{\Delta_L^2}{g}-\frac{2\Delta'\Delta_L}{rg}.
$$
Thus we obtain the formula for $P'(e^{L}\ch')$. This shows part (B).
\end{proof}


Suppose we are in the situation of Lemma \ref{lem:Bertram}(A).  By (\ref{eq-F}),
\begin{eqnarray*}
F(\ch)&=&\frac{\delta}{g}\left(w-\frac{y_2}{x}\right)^2+\frac{1}{x^2 g}(\ch_1^2-2\ch_0\ch_2-\Delta^2),\\
F(e^L\ch)&=&\frac{\delta}{g}\left(w-\frac{y_2}{x}-l_2\right)^2+\frac{1}{x^2 g}(\ch_1^2-2\ch_0\ch_2-(\Delta+x\Delta_L)^2).
\end{eqnarray*}
Since $H\Delta =0$ by assumption, the Hodge Index Theorem implies that $-\Delta^2\geq 0$, and  equality holds if and only if $\Delta=0$.  Similarly, we have  $-(\Delta+x\Delta_L)^2\geq 0$ with equality if and only if   $\Delta+x\Delta_L=0$.  Therefore,  if $\ch$ is of Bogomolov type, then $F(\ch)\geq 0$ and $F(e^L(\ch))\geq 0$ for all $w$. Thus the points $P(\ch)$ and $P(e^L\ch)$ are on or below the parabola $q=\frac{s^2}{2}$. If we are in the situation of Lemma \ref{lem:Bertram}(B), then a similar argument works for $P'(\ch')$ and $P'(e^{L}\ch')$ provided $\ch'$ is of Bogomolov type.

\section{Potential walls in $(\lambda,0,0,q)$-plane for one-dimensional objects}\label{sec:potentialwallsone-dim} We give a parallel result of \ref{prop:Potentialwallsinlambda-qplane} for potential walls in the $(\lambda,0,0,q)$-plane in the case of  1-dimensional objects. We use the notation in \ref{para:intersectionnumbers}.

Fix $\ch$ with $\ch_0=0$ and $\ch_1H_\lambda>0$. Let $\ch'$ be a destabilizing character. So $\ch_0'\neq 0$. We have $\ch'=e^{L}\left(\ch_0',0,\ch_2'-\frac{{\ch_1'}^2}{2\ch_0'}\right)$, and $\ch=\left(0,\ch_1,\ch_2\right)=e^{L}\left(0,\ch_1,\ch_2-L\ch_1\right)$ with $L=\frac{\ch_1'}{\ch_0'}$.

\begin{prop}[Potential walls in $(\lambda,0,0,q)$-plane for one-dimensional objects]\label{prop:Potentialwallsinlambda-qplane-dim-1}
Let
$$\ch=\left(0,\ch_1,\ch_2\right)=(0,k\Theta+pf+\sum_i \xi_i\Theta_i,z).$$
Take the frame $(H_\lambda, H_\lambda^\perp,0)$ as \eqref{eq-frame-elliptic} with $w=0$.
Suppose $\ch_1 H_\lambda>0$ where $H_\lambda$ is given by \eqref{def:H_lambda}. 
Write
$$\quad L=a_L\Theta + b_L f + \sum_i \eta_i \Theta_i$$
for some $a_L, b_L, \eta_i \in \mathbb{R}$. 
Consider the $(\lambda,0,0,q)$-plane in $\Stab (X)$ in \ref{para:lambda-q-plane}.
Then the potential wall $W(e^{L}\ch, e^{L}\ch')$ with
$$\ch'=\left(\ch_0',\ch_1',\ch_2'\right)=(r, 0,\chi),$$
has the following asymptotic behavior in the $(\lambda,0,0,q)$-plane as $\lambda\to 0^+$. 
Write $e=-\Theta^2$, $\Delta_i$ as \eqref{eqn:Delta_i}, $\Delta_L=\sum_i \eta_i \Delta_i$ and $\Delta=\sum_i \xi_i \Delta_i$.
\begin{itemize}
\item[(A)] Suppose $k+\sum_i\xi_i=0$ and $p-ek+\sum_i\xi_i \theta_i\neq 0$.  Set
 \begin{align}
  A:=-\left(z+\Delta\Delta_L+(a_L+\sum_i\eta_i)\Big(p-\frac{e}{2}k+\sum_i\xi_i(\theta_i+\frac{e}{2})\Big)\right)\frac{a_L+\sum_i\eta_i}{p-ek+\sum_i\xi_i \theta_i},  \label{eq:A-dim-1}\\
  B:=\frac{\chi}{r}+\frac{\Delta_L^2}{2}-\left(b_L-ea_L+\frac{e}{2}\Big(a_L+\sum_i\eta_i\Big)+\sum_i\eta_i\theta_i\right)\frac{z+\Delta\Delta_L}{p-ek+\sum_i\xi_i \theta_i}.  \label{eq:B-dim-1}
  \end{align}
  \begin{itemize}
  \item[(A1)] If  $A \neq 0$ then the potential wall is asymptotic to
  \begin{equation}\label{eq:wall-asymptotic-type-B1-dim-1}
  q=\frac{A}{2\lambda^2}.
  \end{equation}
  \item[(A2)] If $A = 0$ and $B \neq 0$ then the potential wall is asymptotic to
  \begin{equation}\label{eq:wall-asymptotic-type-A2-dim-1}
  q=\frac{B}{2\lambda}.
  \end{equation}
  \item[(A3)] If $A=0$ and $B= 0$, then the potential wall is bounded as $\lambda\to 0^+$.
  \end{itemize}
\item[(B)] Suppose $k+\sum_i\xi_i\neq 0$.  Set
\begin{equation}\label{eq:D-dim-1}
  D:=\frac{\chi}{r}+\frac{\Delta_L^2}{2}-\left(z+\Delta\Delta_L+\Big(a_L+\sum_i\eta_i\Big)\Big(p-\frac{e}{2}k+\sum_i\xi_i(\theta_i+\frac{e}{2})\Big)\right)\frac{a_L+\sum_i\eta_i}{k+\sum_i\xi_i}.
  \end{equation}
  \begin{itemize}
  \item[(B1)] If $D \neq 0$   then the potential wall is asymptotic to
  \begin{equation}\label{eq:wall-asymptotic-type-C1-dim-1}
  q=\frac{D}{2\lambda}.
  \end{equation}
  \item[(B2)] If $D= 0$, then the potential wall is bounded as $\lambda\to 0^+$.
  \end{itemize}
\end{itemize}
\end{prop}
\begin{proof} The proof is similar as the proof of Proposition \ref{prop:Potentialwallsinlambda-qplane}. 
We are taking $H_\lambda$ as \eqref{def:H_lambda} and taking $H=H_\lambda$ in \eqref{eq-ch-Maciocia}. 
Recall that $g$ is given in \eqref{eq:g_in_H_lambda_coordinate}.
Now $y_1g=\ch_1H_\lambda>0$ by the assumption. So $y_1>0$.
The potential wall $W(e^{L}\ch, e^{L}\ch')$ in the $(\lambda,0,0,q)$-plane is given by
\begin{equation}\label{eq:wall-lambda-q-general-dim-1}
q=-\left(z+\Delta\Delta_L\right)\frac{gl_1}{gc_1}\frac{1}{g}+l_1l_2\left(\frac{y_1+y_2}{y_1}\right)-\frac{1}{2}(l_1+l_2)^2+\frac{\chi}{rg}+\frac{\Delta_L^2}{2g}.
\end{equation}
Similar computation shows that
\begin{eqnarray*}
q&=&-\left(z+\Delta\Delta_L\right)\frac{(a_L+\sum_i\eta_i)+\Big(b_L-ea_L+(a_L+\sum_i\eta_i)(m-1)+\sum_i\eta_i\theta_i\Big)\lambda}{(k+\sum_i\xi_i)+\Big(p-ek+(k+\sum_i\xi_i)(m-1)+\sum_i\xi_i\theta_i\Big)\lambda} \cdot \frac{1}{g}\\
& &+\frac{1}{4}\left(\left(\frac{b_L-\frac{e}{2}a_L+\sum_i (\theta_i+\frac{e}{2})}{1+(m-\frac{e}{2}-1)\lambda}\right)^2-\frac{(a_L+\sum_i \eta_i)^2}{\lambda^2}\right)\\
& &\quad \cdot \frac{2\lambda\Big(p-\frac{e}{2}k+\sum_i\xi_i(\theta_i+\frac{e}{2})\Big)}{(k+\sum_i\xi_i)+\Big(p-ek+(k+\sum_i\xi_i)(m-1)+\sum_i\xi_i\theta_i\Big)\lambda} \\
& & -\frac{1}{2}\left(\frac{b_L-\frac{e}{2}a_L+\sum_i (\theta_i+\frac{e}{2})}{1+(m-\frac{e}{2}-1)\lambda}\right)^2+\left(\frac{\chi}{r}+\frac{\Delta_L^2}{2}\right)\cdot\frac{1}{g}.
\end{eqnarray*}
The proof follows from the asymptotic analysis of above formula as $\lambda\to 0^+$.
\end{proof}

\bibliography{refs}{}
\bibliographystyle{plain}

                                                              \end{document}